\documentclass[11pt]{amsart}
\usepackage{amssymb}

\newcommand{\R}{{\mathbb R}}

\newcommand{\al}{\alpha}

\newcommand{\vp}{\varphi}
\newcommand{\w}{\mathbf{W}_{1,p}}
\newcommand{\wa}{\mathbf{W}_{\al,p}}
\newcommand{\M}{M^{+}(\R^n)}
\begin{document}

\title[Nonlinear Elliptic Equations and Potentials of Wolff Type]{Nonlinear Elliptic Equations\\ and Intrinsic Potentials of Wolff Type}
\author{Cao Tien Dat}
\author{ Igor Verbitsky}


\thanks{Supported in part by
NSF grant DMS-1161622.}


\newtheorem{theorem}{Theorem}[section]
\newtheorem{lemma}[theorem]{Lemma}
\numberwithin{equation}{section}
\newtheorem{Remark}[theorem]{Remark}
\newtheorem{cor}[theorem]{Corollary}
\newtheorem{Corollary}[theorem]{Corollary}
\newtheorem{prop}[theorem]{Proposition}
\newtheorem{Definition}[theorem]{Definition}

\begin{abstract} 
We give necessary and sufficient conditions for the existence of weak  solutions   
to the model  equation 
$$-\Delta_p u=\sigma \, u^q \quad \text{on} \, \, \, \R^n,$$ 
in the  case  $0<q<p-1$, where $\sigma\ge 0$ is an arbitrary locally integrable function,  or measure, and $\Delta_pu={\rm div}(\nabla u|\nabla u|^{p-2})$ is the $p$-Laplacian.  Sharp  global pointwise estimates and regularity properties of solutions are obtained as well. 
As a consequence, we characterize the solvability of the equation 
\begin{equation*}
-\Delta_p v \, =  {b} \,  \frac {|\nabla v|^{p}}{v} + \sigma \quad \text{on} \, \, \, \R^n, 
\end{equation*} 
where ${b}>0$. These results are new even in the classical case $p=2$. 

Our approach is based on the use of special nonlinear potentials of Wolff type adapted for ``sublinear'' problems, and related integral  inequalities. It allows us to treat simultaneously several problems of this type, such as equations with general quasilinear operators $\text{div} \,  \mathcal{A}(x, \nabla u)$, 
fractional Laplacians $(-\Delta)^{\alpha}$, or fully nonlinear $k$-Hessian operators.   
\end{abstract}

\maketitle

\section{Introduction}

In the present paper, we study elliptic equations of the type 
\begin{equation}\label{eq1}
\left\{ \begin{array}{ll}
-\Delta_pu=\sigma \, u^q & \text{in} \, \, \,  {\R}^n,\\&\\
\displaystyle{\liminf_{x \to \infty} \, u(x)}=0, & \, \, u>0, 
\end{array} \right.
\end{equation}
where  $0<q<p-1$, $\Delta_p =\text{div}(\nabla u|\nabla u|^{p-2})$ is the $p$-Laplace operator, 
and $\sigma\ge 0$ is an arbitrary  locally integrable function, or locally finite Borel measure, $\sigma \in \M$; if $\sigma \in L^1_{{\rm loc}} (\R^n)$ we write $ d \sigma= \sigma \, dx$. 

Our main goal is to give necessary and sufficient conditions on $\sigma$ for the existence of weak solutions to \eqref{eq1}, understood in an appropriate renormalized sense. We also obtain matching upper and lower global pointwise bounds, and provide sharp $W^{1, p}_{{\rm loc}}$-estimates of solutions.   On our way, we identify key integral 
inequalities, and construct new nonlinear potentials of Wolff type that are intrinsic 
to a number of similar problems.

In particular, our approach is applicable to general quasilinear $\mathcal{A}$-Laplace operators $\text{div} \,  \mathcal{A}(x, \nabla u)$, and fully nonlinear $k$-Hessian operators, 
as well as equations with the fractional Laplacian, 
\begin{equation}\label{fraclap}
\left\{ \begin{array}{ll}
(-\Delta )^{\al}u=\sigma \, u^q & \text{in} \, \, \,  \mathbb{R}^n,\\ &\\
\displaystyle{\liminf_{ x \to \infty} \, u(x)}=0,& \, \, u>0, 
\end{array} \right.
\end{equation} 
for $0<q<1$ and $0<\al < \frac{n}{2}$; this includes the range $\al>1$ where the usual maximum principle is not available.

In the classical case $p=2$,  equation \eqref{eq1}, or equivalently \eqref{fraclap} with $\al=1$,  and $0<q<1$,  
serves as a model {\it sublinear} elliptic problem. It is easy to see that it is equivalent to the integral equation $u =\mathbf{N} (u^q d \sigma)$, where 
${\bf N} \omega=(-\Delta)^{-1} \omega$ is the Newtonian potential of  $d \omega = u^q d \sigma$ on $\R^n$. 

As we emphasize below,  equation \eqref{eq1} with $p=2$ and $0<q<1$ is governed by the important   integral inequality 
\begin{equation}\label{weight-2}
\left(\int_{\R^n} |\varphi|^q \, d \sigma\right)^{\frac 1 q} \le \varkappa \,  ||\Delta \varphi||_{L^1(\R^n)},  
\end{equation} 
for all test functions $\varphi \in C^2(\R^n)$ vanishing at infinity such that $-\Delta \varphi \ge 0$. 

Inequality \eqref{weight-2} represents the end-point case of the well-studied $(L^p, L^q)$ trace inequalities for $p>1$. A comprehensive  treatment  
of trace inequalities can be found  in \cite{Maz11}.

More precisely, we will use a localized version of \eqref{weight-2} where the measure 
$\sigma$ is  restricted to a 
ball $B=B(x, r)$, and the corresponding best constant $\varkappa$ is denoted by $\varkappa(B)$. These constants are used as building blocks in our key tool, a nonlinear potential of Wolff type, 
\begin{equation} \label{potential}
\mathbf{K} \sigma (x)  =  \int_0^{\infty} \frac{ \left[\varkappa(B(x, r))\right]^{\frac{q}{1-q}}}{r^{n-2}}\frac{dr}{r},  \quad x \in \R^n,
\end{equation} 
which, together with the usual Newtonian potential ${\bf N} \sigma$, provides sharp global estimates of solutions in the case $p=2$ and  $0<q<1$.

\par\ \par

This work has been motivated by the results of  Brezis and Kamin \cite{BK} who proved that \eqref{eq1},  with $p=2$ and $0<q<1$, has a {\it bounded} solution $u$ if and only 
if  ${\bf N} \sigma \in L^{\infty}(\R^n)$; moreover,  such a solution is unique, and there exists a constant $c>0$ so that 
$$
c^{-1} \, \left[\mathbf{N} \sigma(x)\right]^{\frac{1}{1-q}} \le u (x) \le c \, \mathbf{N} \sigma (x), \quad x \in \R^n. 
$$ 
As was pointed out in  \cite{BK}, both the lower and upper estimates of $u$ are sharp in a sense.  However, there is a substantial gap between them. We will be able  to bridge 
this gap by using both  $\mathbf{N} \sigma$ and $\mathbf{K} \sigma$, and extend these results to possibly unbounded solutions, as well 
as more general nonlinear equations.   
\par\ \par

We will be referring to  equation \eqref{eq1}  with $1<p<\infty$ and  $0<q<p-1$, as well as other nonlinear problems where analogous phenomena occur in a natural way, 
as {\it sublinear problems} in general. One of the main features that distinguishes them  from the ``superlinear'' case $q\ge p-1$ is the absence of any smallness assumptions on $\sigma$. 

Simultaneously with  \eqref{eq1}, we will be able to investigate  the equation 
 with singular natural growth in the gradient term, 
\begin{equation}\label{ric-eq}
\left\{ \begin{array}{ll}
-\Delta_p v \, =  {b} \,  \frac {|\nabla v|^{p}}{v} + \sigma& \text{in} \, \, \,  {\R}^n,\\ &\\ 
\displaystyle{\liminf_{x \to \infty} \, v(x)=0}, & \, \, v>0, 
\end{array} \right.
\end{equation} 
where $\sigma \ge 0$ as above, and ${b} >0$ is a constant that can be expressed in terms of $q$ in  \eqref{eq1},  
\begin{equation}\label{constB}
{b} =\frac{q(p-1)}{p-1-q}, \qquad 0<q<p-1. 
\end{equation} 

 Equations \eqref{eq1} and  \eqref{ric-eq}  are {\it formally} related via the transformation 
 \begin{equation}\label{substitution}
v= \frac{p-1}{p-1-q} \, u^{\frac{p-1-q}{p-1}}.
\end{equation} 

Actually, this relationship fails for some solutions $u$ and $v$ due to the occurrence of certain singular measures, 
as was first observed by Ferone and Murat \cite{FM} (see also \cite{GM}) who studied a similar phenomenon for a related class of equations. In general, a solution $v$ of  
\eqref{ric-eq} gives rise merely  to a supersolution $u$ of \eqref{eq1}. 
Nevertheless, a careful analysis  allows us to give necessary and sufficient conditions for the existence of weak solutions to \eqref{ric-eq}, and justify this transformation   whenever possible (see Theorem~\ref{main-riccati} and the discussion in Sec.~\ref{gradeq} below).

\par\ \par

Equations of the type \eqref{eq1} and  \eqref{ric-eq} have been extensively studied,  mostly in bounded domains $\Omega \subset \R^n$, with $\sigma \in L^r (\Omega)$ for some $r>1$, in \cite{Kra, BrOs, BO1, ABLP, AHBV, AGPW}.  
Various existence and uniqueness results 
for solutions in certain Sobolev spaces, and further references,  can be found there.

However, the precise conditions on $\sigma$ which ensure the existence of solutions are more subtle. In particular, $\sigma$ 
can be an $L^1_{{\rm loc}}$-function, or a measure singular 
with respect to Lebesgue measure. (Notice that $\sigma$ must be absolutely continuous with respect 
to the $p$-capacity; see Lemma~\ref{abscap} below.)  
Analogues of our results for bounded domains $\Omega$ under minimal restrictions on $\sigma$ 
will be presented elsewhere.

\par\ \par

We now introduce  some elements of nonlinear potential theory that will be used throughout the paper. Originally Wolff  potentials were introduced  in \cite{HW} in relation to the spectral synthesis problem in Sobolev spaces. The Wolff potential $\wa\sigma$,   
 where $\sigma\in \M$,  is defined, for $1<p<\infty$ and  $0<\al< \frac n p$, by 
\begin{equation} \label{Wolffpot}
\wa\sigma(x)=\int_0^{\infty} \left[\frac{\sigma(B(x,r))}{r^{n-\al p}}\right]^{\frac{1}{p-1}}\frac{dr}{r}, \quad x \in \R^n.  
\end{equation}
Here $\sigma(B(x, r))=\int_{B(x, r)} d \sigma$ for a ball $B(x, r)=\{y \in \R^n: \, |x-y|<r\}$. 

 In the context of quasilinear problems, Wolff potentials with $\al=1$ appeared in the fundamental work of Kilpel\"{a}inen and Mal\'{y} \cite{KM1, KM2}. 
 A global version of one of their main theorems states that if $U \ge 0$ is a solution  to the equation 
\begin{equation}\label{Ueq}
\left\{ \begin{array}{ll}
-\Delta_p U  =\sigma &\textrm{in ~~} \mathbb{R}^n,\\ &\\
\displaystyle{\, \, \, \liminf_{x \to \infty} \, U(x)=0},
\end{array} \right.
\end{equation}
understood in a potential theoretic or renormalized sense (see \cite{KKT}), then there exists a constant $K=K(n,p)>0$  such that 
\begin{equation}\label{Uest}
 \frac{1}{K} \, \mathbf{W}_{1, p} \sigma (x) \le U(x)  \le K  \, \mathbf{W}_{1, p} \sigma (x), \quad x \in  \mathbb{R}^n. 
\end{equation}
Moreover, a solution $U\ge 0$  to \eqref{Ueq} exists if and only if $1<p<n$, and $ \mathbf{W}_{1, p} \sigma \not\equiv +\infty$, or equivalently  (see \cite{PV1}), 
\begin{equation}\label{Wolfinite}
\int_1^{\infty} \left[\frac{\sigma(B(0, r))}{r^{n- p}}\right]^{\frac{1}{p-1}}\frac{dr}{r} < +\infty.
\end{equation}

It turns out that Wolff potentials alone are not enough to control solutions of \eqref{eq1}. Along with $\mathbf{W}_{1, p} \sigma$, we will use intrinsic  
potentials of Wolff type  associated with  the  localized weighted norm inequalities,    
\begin{equation} \label{weight-lap}
\left(\int_{B} |\varphi|^q \, d \sigma\right)^{\frac 1 q} \le \varkappa(B) \,  ||\Delta_p \varphi||^{\frac{1}{p-1}}_{L^1(\R^n)},  
\end{equation} 
for all test functions $\varphi$ such that  $-\Delta_p \varphi \ge 0$, $\displaystyle{\liminf_{x \to \infty}} \, \varphi(x)=0$. Here $\varkappa(B)$ denotes the best constant in \eqref{weight-lap} 
associated with the measure $\sigma_B=\sigma|_B$ restricted to a ball $B$.  

We now introduce a new nonlinear potential  $\mathbf{K}_{1, p}\sigma$ defined by
\begin{equation} \label{potentialK}
\mathbf{K}_{1, p} \sigma (x)  =  \int_0^{\infty} \left[\frac{ \varkappa(B(x, r))^{\frac{q(p-1)}{p-1-q}}}{r^{n- p}}\right]^{\frac{1}{p-1}}\frac{dr}{r}, \quad x \in \R^n.
\end{equation} 
As we will show below,  $\mathbf{K}_{1, p} \sigma \not\equiv + \infty$ if and only if 
\begin{equation} \label{suffcond1}
\int_1^{\infty} \left[\frac{ \varkappa(B(0, r))^{\frac{q(p-1)}{p-1-q}}}{r^{n- p}}\right]^{\frac{1}{p-1}}\frac{dr}{r} < \infty.  
\end{equation}

We state our main results for equation \eqref{eq1}  in the form of the following theorems. Note that weak solutions $u \in L^q_{{\rm loc}} (d \sigma)$ are understood in the renormalized, or potential theoretic
sense (see Sec.~\ref{pre} for definitions); if $u \in W^{1, p}_{{\rm loc}}(\R^n)$, then  they are the usual distributional solutions. 
\begin{theorem} \label{mainthm}
Let $1<p<n$, $0 < q < p-1,$ and let $\sigma \in \M$. 

{\rm (i)} If both  \eqref{Wolfinite} and \eqref{suffcond1} hold, then  there exists a minimal 
renormalized ($p$-superharmonic) solution $u>0$ to \eqref{eq1} such that  
\begin{equation} \label{two-sided}
c^{-1}   \left [\mathbf{K}_{1, p} \sigma + \left(\w \sigma\right)^{\frac{p-1}{p-1-q}}\right] 
\le u \le c  \left [\mathbf{K}_{1, p} \sigma + \left(\w \sigma\right)^{\frac{p-1}{p-1-q}}\right],  
\end{equation}
where $c>0$ is a constant which  depends only on $p$, $q$, and $n$.

{\rm (ii)} Conversely, if there exists a nontrivial  renormalized supersolution $u$ to \eqref{eq1}, then both  \eqref{Wolfinite} and \eqref{suffcond1} hold,  and $u$ is bounded below 
by the minimal solution of statement {\rm (i)}. 

{\rm (iii)} In the case $p \ge n$ there are no nontrivial  supersolutions on $\R^n$. 
\end{theorem}

We observe that neither of conditions  \eqref{Wolfinite} or \eqref{suffcond1} implies the other one.  
Condition  \eqref{Wolfinite} alone is not enough to ensure the existence of a global solution $u$ even if all 
the local embedding constants $\varkappa(B(0, r))$ in \eqref{suffcond1} are finite,  unless $\sigma$ is radially symmetric (see a counter example in Sec.~\ref{eg} below). 

\par\ \par

In the next theorem, we characterize solutions with $W^{1, p}_{{\rm loc}}$-regularity.

\begin{theorem} \label{W1pthm} Under the assumptions of Theorem \ref{mainthm}, there exists a nontrivial distributional solution 
$u\in W^{1, p}_{{\rm loc}}(\R^n)$ to \eqref{eq1} if and only if   \eqref{Wolfinite} and \eqref{suffcond1} hold together with the  local condition  
\begin{equation} \label{locWolff}
\int_B\left(\w \sigma_B\right)^{\frac{(1+q)(p-1)}{p-1-q}} d\sigma < \infty,
\end{equation}  
for all balls $B=B(0, r)$ in $\R^n.$ 
\end{theorem}

We remark that conditions \eqref{Wolfinite}, \eqref{suffcond1} and \eqref{locWolff} 
are mutually independent. 

\par\ \par

Nonlinear elliptic PDE discussed above are studied in the general framework of nonlinear integral equations, 
\begin{equation}\label{integral} 
u=\wa(u^qd\sigma) \quad \text{in} \, \, \R^n, \quad u > 0,  
\end{equation} where 
$1<p<\infty$, $0<\al< \frac{n}{p}$. Here the special case $\al=1$ corresponds to the $p$-Laplacian, whereas $\al=\frac{2k}{k+1}$, $p=k+1$ to the $k$-Hessian operator (see \cite{TW0}, \cite{PV1}). 

The special case $p=2$ in \eqref{integral} gives the fractional Laplace equation \eqref{fraclap} in the equivalent integral form $u=\frac{1}{c(\alpha, n)} \, \mathbf{I}_{2 \al}(u^q \, d \sigma),$ where  $\mathbf{I}_{2 \al}\mu$ is the Riesz potential of order $2 \al$: 
$$\mathbf{I}_{2 \al}\mu= |x|^{2 \alpha-n} \star \mu = (n-2 \alpha) \, \mathbf{W}_{\al, 2}\mu= c(\alpha, n) \, (-\Delta)^{-\al} \mu,$$ 
for $\mu \in \M$, $0<2 \al<n$. In what follows,  the normalization 
constant $c(\alpha, n)$ will be dropped for the sake of convenience; in particular, 
$\mathbf{I}_2 \mu={\bf N} \mu$. 

We will introduce in Sec.~\ref{integralequation} a fractional version $\mathbf{K}_{\al, p}$ of the intrinsic potential  \eqref{potentialK} for all $p>1$, $0<q<p-1$, $0<\al p<n$, and 
in Sec.~\ref{thm} deduce analogues of 
Theorem~\ref{mainthm} and Theorem~\ref{W1pthm} for the $\mathcal{A}$-Laplacians, $k$-Hessians and fractional Laplacians as  
a consequence of the general Theorem~\ref{main}.

In particular, for the fractional Laplacian equation \eqref{fraclap}, let 
$\kappa(B)$ denote  the least constant in the localized integral inequality
\begin{equation} \label{Rieszineq}
||\mathbf{I}_{2 \al} \nu||_{L^q(d\sigma_{B})} \le \kappa(B) \, \nu(\R^n), \quad \forall \nu \in \M, 
\end{equation}
where $0<q<1$. It is easy to see 
that the constant $\kappa(B)$ does not change if we restrict ourselves to absolutely 
continuous $\nu \in L^1_{+}(\R^n)$.

We define the corresponding nonlinear potential of Wolff type by 
\begin{equation} \label{potential-frac}
\mathbf{K}_{\al, 2} \sigma (x)  =  \int_0^{\infty} \frac{\left[ \kappa(B(x, r))\right]^{\frac{q}{1-q}}}{r^{n- 2\al}}\frac{dr}{r},  \quad x \in \R^n.  
\end{equation}
Conditions \eqref{Wolfinite}, \eqref{suffcond1} need to be replaced with 
\begin{equation}\label{Rieszfinite}
 \int_1^{\infty} \frac{ \left[\kappa(B(0, r))\right]^{\frac{q}{1-q}}}{r^{n- 2\al }}\frac{dr}{r}  + \int_1^{\infty} \frac{\sigma(B(0, r))}{r^{n- 2 \al}}\frac{dr}{r} < \infty,  
\end{equation}
which ensures that both $\mathbf{K}_{\al, 2} \sigma$ and $\mathbf{I}_{2 \al} \sigma$  are not identically infinite. 

We state our main results for sublinear fractional Laplacian equations as follows. 

\begin{theorem} \label{fracmain}
Let $0 < \al < \frac{n}{2}$,  $0< q<1$, and $\sigma \in \M$.

{\rm (i)} Suppose that \eqref{Rieszfinite} holds. Then there exists a minimal solution $u>0$ to \eqref{fraclap} such that 
$\liminf_{x \to \infty} u(x)=0$, and 
\begin{equation} \label{fracsolution}
c^{-1}   \left [\mathbf{K}_{\al, 2} \sigma + \left(\mathbf{I}_{2 \al} \sigma\right)^{\frac{1}{1-q}}\right] 
\le u \le c  \left [\mathbf{K}_{\al, 2} \sigma +  \left(\mathbf{I}_{2 \al} \sigma\right)^{\frac{1}{1-q}}\right], 
\end{equation} 
where $c>0$ is a constant which depends only on $\al$, $q$, and $n$.

{\rm (ii)} Conversely, if there exists a nontrivial supersolution $u$ to \eqref{fraclap} then \eqref{Rieszfinite} holds, and $u$ satisfies the lower bound in \eqref{fracsolution}.
\end{theorem}

It is worth observing that condition \eqref{Rieszfinite} characterizes the existence 
of $0<u<\infty$ $d \sigma$-a.e. such that $u \ge \mathbf{I}_{2 \alpha} (u^q d \sigma)$ 
$d \sigma$-a.e. on $\R^n$, for $0<q<1$, 
which can be regarded as a sublinear version of Schur's Lemma in this case.

\par\ \par

We next turn to equation  \eqref{ric-eq} treated via relation \eqref{substitution}. 
The following theorem demonstrates that conditions  
\eqref{Wolfinite} and \eqref{suffcond1} are necessary and sufficient for the solvability 
of  this equation as well. In particular, if $u$ is  a solution  to \eqref{eq1} then $v$ is a solution to \eqref{ric-eq}. The opposite implication fails to be  true  since $u$ is only a supersolution to \eqref{eq1}. In order that  $u$ be a genuine solution, one needs to impose extra restrictions on $v$ specified in statement (iii) of Theorem~\ref{main-riccati}. These restrictions  are sharp as is evident from simple examples (see details in Sec.~\ref{gradeq}).

\begin{theorem} \label{main-riccati}  
Let $1<p<\infty$ and $0<q<p-1$. Suppose ${b}>0$ is defined by \eqref{constB}, and $\sigma \in M^{+}(\R^n)$. 
 
 (i) If $u$ is a renormalized solution to \eqref{eq1} then $v$ defined by \eqref{substitution} is a renormalized  solution to \eqref{ric-eq}. Consequently, if both \eqref{Wolfinite} and \eqref{suffcond1} hold, then \eqref{ric-eq}  has a renormalized  solution   $v$ 
 which satisfies 
both the lower bound 
\begin{equation} \label{lowerB}
v \ge c^{-1} \left[ \left(\mathbf{K}_{1, p} \sigma\right)^{\frac{p-1-q}{p-1}} +  \mathbf{W}_{1, p} \sigma
\right], 
 \end{equation} and the upper bound 
\begin{equation} \label{upperB}
v \le c \left[ \left(\mathbf{K}_{1, p} \sigma\right)^{\frac{p-1-q}{p-1}} +  \mathbf{W}_{1, p} \sigma\right], 
\end{equation}
where $c>0$ depends only on $p$, $q$, and $n$. 

(ii) If   \eqref{ric-eq} has a renormalized solution 
$v > 0$, then for every ball $B$ and $w_B=\frac{|Dv|^p}{v} \, \chi_B$, we have 
\begin{equation} \label{necessary}
||v||_{L^{\frac{q(p-1)}{p-1-q}, \infty}(w_B)} < \infty. 
\end{equation}
Moreover, $v$ satisfies the lower bound  \eqref{lowerB}, and $u$ defined by \eqref{substitution} is a renormalized supersolution to 
\eqref{eq1}; consequently, both 
\eqref{Wolfinite} and \eqref{suffcond1} hold. 

(iii) Furthermore, if $v$ satisfies a strong-type version of \eqref{necessary}, 
\begin{equation} \label{sufficient}
||v||_{L^{\frac{q(p-1)}{p-1-q}}(w_B)} < \infty,
\end{equation}
for every ball $B$, then $u$ is actually a renormalized solution to \eqref{eq1}.
 \end{theorem}

In conclusion, we remark that we have stated our results for minimal  
``ground state'' solutions which vanish at infinity, but they have obvious analogues  for 
solutions such that $\liminf_{x \to \infty}u=c$ where $c > 0$, as discussed in  \cite{BK} 
in the case $p=2$. 

The brief contents of the paper are as follows. In Sec. \ref{pre} we introduce basic definitions, notations, and preliminary results concerning quasilinear equations and nonlinear potentials. In Sec. \ref{lower}, we obtain some useful estimates 
involving Wolff potentials $\wa \sigma$. 
The corresponding nonlinear integral equations \eqref{integral} and properties 
of the intrinsic Wolff potentials $\mathbf{K}_{\alpha, p} \sigma$ 
are studied in Sec. \ref{integralequation}. In Sec. \ref{thm}, we prove our main theorems regarding equation \eqref{eq1}, and discuss briefly more general quasilinear and fully nonlinear equations. In  Sec. \ref{eg}, we give a counter example which demonstrates that merely the finiteness of the embedding constants $\varkappa(\sigma_B)$ is not enough for the existence of a global solution to \eqref{eq1}, even if $\mathbf{W}_{1, p} \sigma <\infty$ a.e. Sec. \ref{gradeq} is devoted to  equations with singular gradient terms \eqref{ric-eq}.

\section{Preliminaries}\label{pre}

Let  $\Omega$ be an open set in $\R^n$, we denote by 
$M^+(\Omega)$ the class of all nonnegative locally finite Borel measures on $\Omega$. We denote the $\sigma$-measure 
of a measurable set $E\subset \Omega$  by 
$
\sigma(E)= |E|_\sigma = \int_E d \sigma.
$

For $p>0$ and $\sigma \in  M^+(\Omega)$, we denote by $L^p(\Omega, d\sigma)$ ($L_{\textrm{loc}}^p(\Omega, d\sigma)$, respectively) the space of measurable functions $f$ such that $|f|^p$ is integrable (locally integrable) with respect to $\sigma$. For $f \in L^p(\Omega, d\sigma)$, we set 
$$
||f||_{L^p(\Omega, d \sigma)} = \Big(\int_\Omega |f|^p \, d \sigma\Big)^{\frac{1}{p}}. 
$$
When $d\sigma=dx$, we write $L^p(\Omega)$ (respectively $L^p_{\textrm{loc}}(\Omega)$), and 
denote Lebesgue measure of $E\subset \R^n$ by $|E|$.  

The Sobolev space $W^{1, p}(\Omega)$ ($W^{1, p}_{\textrm{loc}}(\Omega)$, respectively) is the space of all functions $u$ such that $u\in L^p(\Omega)$ and $|\nabla u|\in L^p(\Omega)$ ($u\in L^p_{\textrm{loc}}(\Omega)$ and $|\nabla u|\in L^p_{\textrm{loc}}(\Omega)$, respectively). 

Let $L^{1, p}_{0}(\Omega)$ ($1<p<n$) denote the homogeneous Sobolev (Dirichlet) space,  
i.e., the closure of $C^\infty_0(\Omega)$ with respect to the norm 
$||u||_{1, p}=||\nabla u ||_{L^p(\Omega)}$ (see, e.g., \cite{MZ1}, Sec. 1.3.4). 

The dual Sobolev space $L^{-1, p'}(\Omega)= L^{1, p}_{0}(\Omega)^{*}$ is the  space of distributions $\nu \in D^{\prime} (\Omega)$ such that 
$$
|| \nu ||_{-1, p'} = \sup \, \frac{\left | \langle u, \nu \rangle \right |}{||u||_{1, p}} < +\infty,
$$ 
where the supremum is taken over all $u \in L^{1, p}_{0}(\Omega)$, $u \not=0$.

We will need Wolff's inequality \cite{HW} (see also \cite{AH}, Sec. 4.5) in the case 
$\Omega=\R^n$ for $\nu \in \M$:  
\begin{equation}\label{wolff-ineq}
c^{-1} \, ||\nu||^{p'}_{-1, p'} \le \int_{\R^n} \w \nu \, d \nu \le c ||\nu||^{p'}_{-1, p'},
\end{equation} 
where 
$1<p< n$, and $c$ is a positive constant which depends only on $n$ and $p$. 
There is a local version of Wolff's inequality (see \cite{AH}, Theorem 4.5.5):
\begin{equation}\label{wolff-ineq-loc}
\nu \in \M\cap W^{-1, p'}_{{\rm loc}}(\R^n) \Longleftrightarrow \int_B \w \nu_B \, d \nu_B<\infty,  \, \, \text{for all balls} \, \,  B,
  \end{equation} 
  where $B=B(x, R)$, and $\nu_B=\nu|_{B}$. 

For $u\in W^{1, p}_{\textrm{loc}}(\Omega)$, we define the $p$-Laplacian $\Delta_p$ ($1<p<\infty$) in the
distributional sense, i.e., for every $\vp \in C_0^{\infty}(\Omega)$,
\begin{equation}\label{distri}
\langle \Delta_p u,   \vp \rangle =\langle \textrm{ div}(|\nabla u|^{p-2}\nabla u, \vp \rangle =
- \int_{\Omega}|\nabla u|^{p-2}\nabla u\cdot \nabla \vp \, dx.
\end{equation}

We will extend the usual distributional definition of solutions $u$ of $-\Delta_p u = \mu$, where $\mu \in W^{-1, p'}_{\textrm{loc}}(\Omega)$, 
to $u$ not necessarily in $W^{1, p}_{\textrm{loc}}(\Omega)$. We will understand 
solutions in the following potential-theoretic sense using $p$-super-harm\-onic functions, which is equivalent to the notion of 
locally renormalized solutions in terms of test functions (see \cite{KKT}).

A function $u \in W^{1, p}_{\textrm{loc}}(\Omega)$ is called $p$-harmonic if it satisfies the homogeneous equation
$\Delta_p u=0$. Every $p$-harmonic function has a continuous representative which coincides with $u$ a.e. 
(see \cite{HKM}). Then $p$-superharmonic functions are defined via a comparison principle. We say that $u\!: \Omega \rightarrow (-\infty, \infty]$ is $p$-superharmonic if $u$ is lower semicontinuous, is not identically infinite in any component of $\Omega$, and, whenever $D\Subset \Omega$ and $h \in C( \overline{D})$ is $p$-harmonic in $D$ with $h \leq u$ on $\partial D$, then $h \leq u$ in $D$. 

A $p$-superharmonic function $u$ does not necessarily belong to $\mathrm{W}^{1, p}_{\textrm{loc}}(\Omega)$, but its truncations $T_k(u) = \min (k, \max(u, -k))$ do, for all $k>0$.   In addition, $T_k(u)$ are supersolutions, i.e., $-\textrm{ div}(|\nabla T_k(u)|^{p-2}\nabla T_k(u)) \geq 0$, in the distributional sense. The generalized gradient of a $p$-superharmonic function $u$ defined by \cite{HKM}: 
$$Du = \lim_{k\rightarrow \infty} \nabla (T_k(u)).$$

We note that every $p$-superharmonic function $u$ has a quasicontinuous representative which coincides with $u$ quasieverywhere (q.e.), i.e., 
everywhere except for a set of $p$-capacity zero (see \cite{HKM}). Here the $p$-capacity 
is defined, for compact sets $E\subset \R^n$,  by 
\begin{equation}\label{p-capacity}
{\rm cap}_{p}(E)= \inf \left \{ ||\nabla u||^p_{L^{p}(\R^n)} : \, \, u \ge 1 \, \, \text{on} \, \, E, \quad 
u \in C^{\infty}_0(\R^n)\right\}.
\end{equation}
We will assume that $u$ is always chosen to be quasicontinuous.

Let $u$ be $p$-superharmonic, and let $1\leq r < \frac{n}{n-1}$. Then $|Du|^{p-1}$, and consequently $|Du|^{p-2}Du$, 
belongs to $L^r_{\textrm{loc}}(\Omega)$ \cite{KM1}.  This allows us to define a nonnegative distribution $-\Delta_p u$ for each  $p$-superharmonic function $u$ by 
\begin{equation}
-\langle  \Delta_p u,   \vp \rangle = \int_{\Omega} |D u|^{p-2} D u\cdot \nabla \vp \, dx, 
\end{equation}
for all $\vp \in C^{\infty}_{0}(\Omega)$.  Then by the Riesz representation theorem there exists a unique measure $\mu[u] \in M^{+}(\Omega)$  so that $-\Delta_p u = \mu[u]$, where 
$\mu[u]$ is called 
the Riesz measure of $u$. 

For $\omega \in M^+(\Omega)$, consider the equation 
$$- \Delta_p  u = \omega \quad {\rm in} \, \, \Omega.$$ 
 Solutions to such equations with measure data are generally understood in the 
potential-theoretic sense (see \cite{KM1}, \cite{KM2}, \cite{K2}). 

\begin{Definition}\label{psupersense} For $\omega \in M^{+}(\Omega)$, 
$u$ is said to be a ($p$-superharmonic) solution to the equation 
\begin{equation}\label{omega-eq}
- \Delta_p  u = \omega \quad {\rm in} \, \, \Omega
\end{equation} 
 if $u$ is $p$-superharmonic in $\Omega$, and $\mu[u] = \omega$.  

Thus, if $\sigma \in \M$, then $u\ge 0$ is a solution to  the equation 
\begin{equation}\label{sigma-eq}
-\Delta_p u = \sigma u^q \quad {\rm in} \, \, \Omega
\end{equation} 
if $u$ is $p$-superharmonic in $\Omega$, 
$u \in L^q_{{\rm loc}} (\Omega, d \sigma)$, and $d \mu[u] = u^q \, d \sigma$. 
\end{Definition} 

Alternatively, we will use the framework of \textit{locally renormalized} solutions. This 
notion  introduced by Bidaut-V\'eron \cite{BV}, following the development of the theory 
of renormalized solutions in \cite{DMMOP}, is well suited for our purposes. 
As was shown recently in \cite{KKT}, for $\omega \in M^+(\Omega)$ it coincides with the notion of a $p$-superharmonic solution in Definition~\ref{psupersense}.

In particular, a $p$-superharmonic function 
$u\ge 0$ satisfying \eqref{omega-eq} 
is a locally renormalized solution 
defined in terms of test functions (see \cite{KKT}, Theorem 3.15). 
 This means that, for all $\vp \in C_0^{\infty}(\Omega)$ and $h \in W^{1, \infty}(\Omega)$ with $h'$ having compact support, we have 
\begin{equation} \label{loc-renorm}
\int_{\Omega} |Du|^{p}\, h'(u) \, \vp \, dx +  \int_{\Omega} |D u|^{p-2} D u \cdot \nabla \vp \, h(u)\, dx=  \int_{\R^n} h(u) \vp \, d \, \omega.
\end{equation}
The converse is also true, i.e., if $u$ is a locally renormalized solution to \eqref{omega-eq} 
then there exists a $p$-superharmonic representative $\tilde u = u$ a.e. 

We will call such solutions of \eqref{loc-renorm} with $d \omega= u^q d \sigma$ (locally) renormalized, $p$-superharmonic, 
or simply solutions, of \eqref{sigma-eq}.

\begin{Definition}\label{psupersol} A function $u\ge 0$ is called a (renormalized) supersolution  to \eqref{sigma-eq} if 
$u$ is $p$-superharmonic in $\Omega$, $u \in L^q_{{\rm loc}} (\Omega, d \sigma)$, and 
\begin{equation}\label{supersoleq}
\int_{\Omega} |D u|^{p-2} D u\cdot \nabla \vp \, dx \ge \int_{\Omega} u^q \vp \, d \sigma, \quad \forall \, \vp \in C_0^{\infty}(\Omega), \quad \vp \ge 0. 
\end{equation}
\end{Definition}

As we will show below, supersolutions to \eqref{eq1} in the sense of Definition~\ref{psupersol} are closely related to supersolutions associated with the integral equation 
\eqref{integral}, i.e., $u \in L^q_{{\rm loc}}(\R^n, d \sigma)$ such that  
\begin{equation} \label{integraleq}
u\ge \wa(u^q \, d\sigma) \quad d \sigma\text{-}{\rm a.e.}, 
\end{equation}
in the case $\al=1$. 
The following weak continuity result will be used to prove the existence of $p$-superharmonic solutions to quasilinear equations.
\begin{theorem}[\cite{TW1}] \label{weakconv} Suppose $\{u_j\}$  are nonnegative 
$p$-super\-harmonic functions in an open set $\Omega$ such that $u_j \to u$ a.e., where $u$ is 
$p$-superharmonic  in $\Omega$. Then $\mu[u_j]$ converges weakly to $\mu[u]$, i.e.,
 for all $\vp \in C_0^{\infty}(\Omega)$, 
$$\lim_{j \to \infty}\int_{\Omega} \vp \, d \mu[u_j] = \int_{\Omega}\vp \, d \mu[u].$$
\end{theorem}

The next theorem is concerned with pointwise estimates of nonnegative $p$-superhamonic functions in terms of Wolff potentials.

\begin{theorem}[\cite{KM2}] \label{thmpotest} Let $1<p< \infty$, and let $u$ be a $p$-superharmonic function in $\R^n$ with $\liminf_{x \to \infty} u=0$. 

(i) If $p< n$ and $\omega=\mu[u]$, then 
\begin{equation}\label{wolff-K}
 \frac{1}{K} \, \w\omega (x) \le u(x) \le K \, \w\omega(x), \quad x \in \R^n,
 \end{equation}
where $K$ is a positive constant depending only on $n$ and  $p$. 

(ii) In the case $p\ge n$, it follows that $u\equiv 0$. 
\end{theorem}

\section{Wolff potential estimates}\label{lower}

We start with some useful estimates for Wolff potentials. Throughout this paper we will assume that $\sigma \in \M$, i.e., $\sigma$ is a locally finite Borel measure on $\R^n$, 
and $\sigma \not= 0$. 

\begin{lemma}\label{average-wolff}  Suppose  $1<p<\infty$,   $0<\al <\frac n p$, and $\sigma \in M^+(\R^n)$. 
Let $s = \min \, (1, p-1)$. Then there exists a positive constant $c$ which depends only on $n$, $p$, and $\al$ 
such that, for all $x \in \R^n$ and $R>0$, 
\begin{equation}\label{average} 
 \begin{aligned}
c^{-1}  &
\int_R^\infty   \left(\frac{\sigma(B(x, r))}{r^{n-\al p}}\right)^{\frac{1}{p-1}}\frac{dr}{r} 
\\
 \le & \inf_{B(x, R)} \, \wa \sigma 
\le \left ( \frac{1} {|B(x, R)|} \int_{B(x, R)} \left[\wa \sigma(y)\right]^s \, dy 
\right)^{\frac 1 s} \\ 
\le & \,  c   \int_R^\infty   \left(\frac{\sigma(B(x, r))}{r^{n-\al p}}\right)^{\frac{1}{p-1}}\frac{dr}{r}.  
\end{aligned}
\end{equation}
\end{lemma}

\begin{proof} Without loss of generality we can assume that  $x=0$. We first prove the last estimate in \eqref{average}. Clearly, 
\begin{equation*} 
\frac{1} {|B(0, R)|}   \int_{B(0, R)} \left[\wa \sigma(y)\right]^s \,  dy  \le I_1 + I_2,
\end{equation*}
where 
\begin{equation*} 
 \begin{aligned}
I_1& = \frac{1} {|B(0, R)|}   \int_{B(0, R)}  \left ( \int_0^R  \left(\frac{\sigma(B(y, r))}{r^{n-\al p}}\right)^{\frac{1}{p-1}} 
  \frac{dr}{r} \right)^s dy, \\ 
  I_2 & = \frac{1} {|B(0, R)|}   \int_{B(0, R)}  \left ( \int_R^\infty  \left(\frac{\sigma(B(y, r))}{r^{n-\al p}}\right)^{\frac{1}{p-1}} 
  \frac{dr}{r} \right)^s dy. 
\end{aligned}
\end{equation*}
To estimate $I_2$, notice that since $B(y, r) \subset B(0, 2r)$ for $y \in B(0, R)$ and $r>R$,  
it follows 
$$
I_2 \le  
\left (\int_R^\infty   \left(\frac{\sigma(B(0, 2r))}{r^{n-\al p}}\right)^{\frac{1}{p-1}}\frac{dr}{r} \right)^s.
$$
To estimate  $I_1$, suppose first that $p \ge 2$ so that $s=1$. Then using Fubini's theorem and 
Jensen's inequality we deduce 
$$ 
I_1 \le  \int_0^R   \left (  \frac{1} {|B(0, R)|} \int_{B(0, R)}   \sigma(B(y, r)) \,  dy \right)^{\frac 1{p-1}} 
 \frac{dr}{r^{\frac{n-\al p}{p-1}+1}}. 
$$ 
Using Fubini's theorem again, we obtain 
$$
 \int_{B(0, R)}   \sigma(B(y, r)) \, dy \le  \int_{B(0, 2R)} |B(y, r)| \, d \sigma = |B(0,1)| \, r^n \, \sigma (B(0, 2R)). 
$$ 
Hence, there is a constant $c=c(n, p, \al)$ such that  
\begin{equation*} 
 \begin{aligned}
I_1 & \le c R^{-\frac{n}{p-1}}\sigma(B(0,2R))^{\frac 1{p-1}} \int_0^R r^{\frac{\al p}{p-1}-1} dr
\\
 &= c R^{\frac{\al p-n}{p-1}} \sigma(B(0,2R))^{\frac 1{p-1}} \le c  \int_R^\infty  
 \left(\frac{\sigma(B(0, 2r))}{r^{n-\al p}}\right)^{\frac{1}{p-1}}\frac{dr}{r}.  
\end{aligned}
\end{equation*}
Notice that this is the same estimate we have deduced above for $I_2$ with $s=1$. 

Let us now estimate $I_1$ for  $1<p<2$ and $s=p-1$. In this case, we will use  
the following elementary inequality: for every $R>0$,  
$$
\left (\int_0^R \left (\frac{\phi({r})}{r^\gamma}\right)^{\frac 1{p-1}} \frac{dr}{r} \right )^{p-1} \le c(p, \gamma) 
 \int_0^{2R} \frac{\phi({r})}{r^\gamma}  \frac{dr}{r}, 
$$
where $\gamma>0$, $1<p<2$, and $\phi$ is a non-decreasing function on $(0, \infty)$. 

Applying the preceding inequality with $\phi ({r})= \sigma(B(0, 2r))$ and $\gamma=n-\al p$, 
and estimating as in the case $p \ge 2$, using Fubini's theorem again, we obtain: 
 \begin{equation*}
 \begin{aligned}
I_1 & \le \frac{c} {|B(0, R)|}   \int_{B(0, R)}   \int_0^{2R}  \frac{\sigma(B(y, r))}{r^{n-\al p}} 
  \frac{dr}{r} dy  \\
  &  \le c \,  R^{-n} \sigma(B(0, 2R)) \int_0^{2R} r^{\al p-1} dr
    = 
    c \, R^{-n+\al p} \sigma(B(0, 2R)) \\
    & \le  c \left (    \int_R^\infty   \left(\frac{\sigma(B(0, 2r))}{r^{n-\al p}}\right)^{\frac{1}{p-1}}\frac{dr}{r}  \right)^{p-1},  
      \end{aligned}
 \end{equation*}
 where $c$ denotes different constants depending only on $n$, $p$, $\al$. 
Combining the estimates for $I_1$ and $I_2$, we arrive at 
$$
 \frac{1} {|B(0, R)|} \int_{B(0, R)} (\wa \sigma)^s \, dy \le c  
\left (\int_R^\infty   \left(\frac{\sigma(B(0, 2r))}{r^{n-\al p}}\right)^{\frac{1}{p-1}}\frac{dr}{r} \right)^s.
$$ 
Making the substitution $\rho=2r$ in the integral on the right-hand side completes the proof of the upper estimate 
in \eqref{average}. 

To prove the lower estimate,  notice that 
$$
\wa \sigma (y) \ge \int_{2R}^\infty  \left(\frac{\sigma(B(y, r))}{r^{n-\al p}}\right)^{\frac{1}{p-1}} 
  \frac{dr}{r} = c  \int_{R}^\infty  \left(\frac{\sigma(B(y, 2 \rho))}{\rho^{n-\al p}}\right)^{\frac{1}{p-1}} 
  \frac{d\rho}{\rho}. 
$$
Since $B(y, 2 \rho) \supset B(0, \rho)$ for $y \in B(0,R)$ and $\rho>R$, there exists $c=c(n, p, \al)>0$ such that 
$$
\inf_{B(0, R)} \, \wa \sigma \ge c  \int_{R}^\infty  \left(\frac{\sigma(B(0, \rho))}{\rho^{n-\al p}}\right)^{\frac{1}{p-1}} 
  \frac{d\rho}{\rho}.
$$ 
\end{proof} 

\begin{Corollary}\label{inf} Suppose  $1<p<\infty$,   $0<\al <\frac n p$, and $\sigma \in M^+(\R^n)$.

(i) $\wa \sigma \not\equiv+\infty$ if and only if 
\begin{equation}\label{wolff01}
 \int_{1}^\infty  \left(\frac{\sigma(B(0, r))}{r^{n-\al p}}\right)^{\frac{1}{p-1}} 
  \frac{dr}{r} < \infty.
  \end{equation} 
  
  (ii) Condition \eqref{wolff01} implies 
  \begin{equation}\label{wolff02}
 \int_{t}^\infty  \left(\frac{\sigma(B(x, r))}{r^{n-\al p}}\right)^{\frac{1}{p-1}} 
  \frac{dr}{r} < \infty, \quad \forall x \in \R^n, \, t>0. 
  \end{equation}

(iii) If \eqref{wolff01} holds, then  $\wa \sigma \in L^s_{{\rm loc}} (dx)$, where $s=\min \, (1, p-1)$, and 
\begin{equation}\label{lim-inf}
\liminf_{|x|\to \infty} \wa \sigma (x) =0.
  \end{equation} 
\end{Corollary} 
\begin{proof} We first verify statement  (ii). Suppose \eqref{wolff01} holds. We may assume  $x\not =0$, since for $x=0$ 
  \eqref{wolff02}  is obvious. Clearly, $B(x, r)\subset B(0, 2r)$ for $|x|<r$, and hence,
$$
I_{x}:=\int_{|x|}^\infty  \left(\frac{\sigma(B(x, r))}{r^{n-\al p}}\right)^{\frac{1}{p-1}} 
  \frac{dr}{r} \le \int_{|x|}^\infty  \left(\frac{\sigma(B(0, 2r))}{r^{n-\al p}}\right)^{\frac{1}{p-1}} 
  \frac{dr}{r} < \infty.
$$
It follows that  \eqref{wolff02} holds for $t\ge |x|$. If $t<|x|$, then 
$$
\int_{t}^\infty  \left(\frac{\sigma(B(x, r))}{r^{n-\al p}}\right)^{\frac{1}{p-1}} 
  \frac{dr}{r}=\int_{t}^{|x|}  \left(\frac{\sigma(B(x, r))}{r^{n-\al p}}\right)^{\frac{1}{p-1}} 
  \frac{dr}{r} + I_{x} < \infty, 
$$
since in the first integral $B(x, r) \subset B(0, 2 |x|)$. Thus, \eqref{wolff02} holds for all $x$ and $t>0$.

It remains to prove \eqref{lim-inf}, since the other statements of 
Corollary~\ref{inf} are immediate from \eqref{average} and \eqref{wolff02}. Suppose that \eqref{wolff01} holds. 
For $R>0$, let $A_R=\{\frac{R}{2}<|x|<R\}$. Then by the upper estimate of Lemma~\ref{average-wolff} (with $x=0$), 
\begin{equation*}
\begin{aligned}
\inf_{|x|> R/2} \,  \wa \sigma(x) & \le\inf_{A_R} \, \wa \sigma(x) \le  \left( \frac{1} {|A_R|}   \int_{A_R} (\wa \sigma)^s \,  dx\right)^{\frac{1}{s}}\\ 
&\le c \, \int_R^\infty   \left(\frac{\sigma(B(0, r))}{r^{n-\al p}}\right)^{\frac{1}{p-1}}\frac{dr}{r},
\end{aligned}
\end{equation*} 
where $c$ does not depend on $R$. Since the right-hand side of the preceding inequality tends to zero as $R\to \infty$, we 
see that  \eqref{lim-inf} holds. 
\end{proof} 

It is easy to see that if $\omega \in M^+(\R^n)$, and $u \in W^{1, p}_{{\rm loc}}(\R^n)$ is a weak solution to the equation 
$-\Delta_p u = \omega$, then  $\omega \in W^{-1, p'}_{{\rm loc}}(\R^n)$. 
The converse statement is less obvious, and we were not able to find it 
in the literature. In the next lemma, for the sake of completeness,  we give a proof 
in the case  $\omega \ge 0$ using a series of Caccioppoli-type inequalities. 

\begin{lemma}\label{loc-sol}  Suppose  $1<p< n$,  and $\omega \in M^+(\R^n)\cap W^{-1, p'}_{{\rm loc}}(\R^n)$. If $u\ge 0$ is a $p$-superharmonic solution to the equation $-\Delta_p u = \omega$ in $\R^n$ 
such that $\liminf_{x\to \infty} u =0$, then $u \in W^{1, p}_{{\rm loc}}(\R^n) \cap L^{1}_{{\rm loc}}(\R^n, d \omega)$.
\end{lemma}

\begin{proof} Let us first show that $u \in L^{1}_{{\rm loc}}(\R^n, d \omega)$ 
using Wolff's inequality \cite{HW}. Fix a ball $B=B(0, R)$, $R>0$.  By Theorem~\ref{thmpotest}, $u$ satisfies the Wolff potential estimate \eqref{wolff-K}. Hence,
\begin{equation*}
\begin{aligned}
\int_B u \, d \omega & \le K \int_B \left (
\int_0^{R}  \frac{\omega (B(x, r))}{r^{n-p}}\right)^{\frac{1}{p-1}}\frac{dr}{r}  d \omega (x)\\
& + K \int_B   \int_{R}^{\infty} \left( \frac{\omega (B(x, r))}{r^{n- p}}\right)^{\frac{1}{p-1}}\frac{dr}{r} d\omega(x) := I + II.  
\end{aligned}
\end{equation*} 

Since $B(x, r) \subset 2B=B(0, 2R)$ for $x \in B$ and $r<R$, we obtain   by \eqref{wolff-ineq-loc}, 
\begin{equation*}
\begin{aligned}
I & \le K \int_B \int_0^{R} \left( \frac{\omega (B(x, r)\cap 2B)}{r^{n-p}}\right)^{\frac{1}{p-1}}\frac{dr}{r}  d\omega(x) \\
& \le K \int_{\R^n} \w \omega_{2B} \, d \omega_{2B}   < \infty. 
\end{aligned}
\end{equation*}

To estimate $II$, notice that $B(x, r)\subset B(0, 2r)$, for $r>R$ and $x \in B$. Hence, 
 $$
I I \le K \omega(B) \, \int_{R}^{\infty} \left(\frac{\omega (B(0, 2r))}{r^{n- p}}\right)^{\frac{1}{p-1}}\frac{dr}{r} < \infty 
$$
by Corollary~\ref{inf}.

We next show that $u \in L^{s}_{{\rm loc}}(\R^n, dx)$ for $0< s\le \frac{np}{n-p}$. 
Arguing as above, we use  \eqref{wolff-K} and split the integral with respect to $dr/r$ 
into two parts: 
\begin{equation*}
\begin{aligned}
\int_B u^s dx & \le c_s\,  K^s \int_B \left (
\int_0^{R} \left( \frac{\omega (B(x, r))}{r^{n-p}}\right)^{\frac{1}{p-1}}\frac{dr}{r} \right)^s dx
 \\
& + c_s \, K^s \int_B  \left( \int_{R}^{\infty} \left( \frac{\omega (B(x, r))}{r^{n- p}}\right)^{\frac{1}{p-1}}\frac{dr}{r}
\right)^s dx := III + IV,  
\end{aligned}
\end{equation*}
where $c$ is a constant depending only on $s$. 

To estimate $III$, notice that by \eqref{wolff-ineq} 
$\omega_{2B} \in  L^{-1, p'}(\R^n)$, and consequently there is a unique solution  $u_{2B}\in L^{1, p}_{0}(\R^n)$  to the equation $-\Delta_p u_{2B} = \omega_{2B}$ in $\R^n$. Hence, by the Sobolev inequality, 
$u_{2B} \in L^s_{{\rm loc}}(\R^n)$ for $0<s\le \frac{np}{n-p}$. Clearly, $u_{2B}$ is 
$p$-superharmonic, and satisfies \eqref{wolff-K} 
with $\omega_{2B}$ in place of $\omega$, i.e., 
$$
 \int_0^{\infty} \left( \frac{\omega (B(x, r)\cap 2B)}{r^{n-p}}\right)^{\frac{1}{p-1}}\frac{dr}{r} \le K \, u_{2B}(x). 
$$  
Since $B(x, r)\subset 2B$ for $x \in B$ and $r<R$, we estimate  
$$
III \le c  \int_B \left (\int_0^{R} \left( \frac{\omega (B(x, r)\cap 2B)}{r^{n-p}}\right)^{\frac{1}{p-1}}\frac{dr}{r} \right)^s dx 
\le c  \int_{B}  u_{2B}^s dx < \infty.
$$

The estimate of $IV$ is similar to that of $II$:  
$$
IV \le c_s \, K \, |B| \left ( \int_{R}^{\infty} \left( \frac{\omega (B(0, 2r))}{r^{n- p}}\right)^{\frac{1}{p-1}}\frac{dr}{r} \right)^s < \infty 
$$
by Corollary~\ref{inf}. Thus, $u \in L^{s}_{{\rm loc}}(\R^n, dx)$ for $s\le \frac{np}{n-p}$.

We next show that there exists $0<\beta\le 1$ such that, for all balls $B$, 
\begin{equation}\label{beta-0}
\int_{B} |D u|^p  u^{\beta - 1} dx < \infty.
\end{equation} 
Indeed, since $u$ is $p$-superharmonic, it is a locally renormalized solution to 
$-\Delta_p u = \omega$ as discussed in Sec.~\ref{pre}. Let $u_k = \min \,(u, k)$, where $k>0$. Note that $u$, and hence 
$u_k$, is locally bounded below. 
Using  $h(u) = u_k^{\beta}$ ($0<\beta\le 1$) in \eqref{loc-renorm}, and a cut-off function $\varphi  \in C^\infty_0(B)$  such that $0\le \varphi \le 1$ and $\varphi =1$ on $\frac{1}{2}B$, we obtain 
\begin{equation}\label{beta-est}
\int_{u \le k} |D u|^p  u^{\beta - 1} \varphi \, dx + \int_{\R^n} |Du|^{p-2} Du \cdot \nabla 
\varphi \, u_k^{\beta} \, dx = \int_B u_k^\beta \varphi \, d \omega. 
\end{equation}

As was shown above, $u \in L^1_{{\rm loc}} (\R^n, \, d \omega)$, 
and hence the right-hand side is bounded by 
\begin{equation}\label{beta-est-tr}
\int_B u^\beta \varphi \, d \omega \le  \omega(B)^{1-\beta} \left (\int_B u  \, d \omega \right)^{\beta} < \infty,
\end{equation}
for $0< \beta\le 1$. 

Since $u$ is $p$-superharmonic, we have $|Du| \in L^{r'(p-1)}$ for $r'<\frac{n}{n-1}$. 
By H\"{o}lder's inequality with exponents $r'$ and $r>n$, we deduce from 
\eqref{beta-est},   
\begin{equation}\label{beta-est2}
\begin{aligned}
\int_{u \le k} |D u|^p  u^{\beta - 1} \varphi \, dx &  \le c \, ||Du||^{p-1}_{L^{r'(p-1)}(B)} 
||u||^{\beta}_{L^{\beta r}(B, dx)}\\
& +  \omega(B)^{1-\beta} \left (\int_B u \, d \omega  \right)^{\beta}. 
\end{aligned}
\end{equation}
If $\beta r  =s\le \frac{np}{n-p}$, where $r>n$ and $\beta\le 1$, then the right-hand side of the preceding inequality is finite. 
Picking $r$ so that $r>n$ and is arbitrarily close to $n$, and passing to the limit as $k \to \infty$, we obtain \eqref{beta-0} for $\beta=\beta_0$, provided 
$$
0<\beta_0<\frac{p}{n-p}, \quad \beta_0 \le 1.
$$ 

In the case $\frac{p}{n-p}>1$, i.e., for $p > \frac{n}{2}$, we can set $\beta_0=1$, which shows that in fact 
$Du \in L^p(\frac{1}{2} B, dx)$, for all $B=B(0, R)$. Hence,  $Du = \nabla u$ in the distributional sense, and consequently $u \in W^{1, p}_{{\rm loc}}(\R^n)$.

For $1<p\le \frac{n}{2}$, we  fix $s$ so that 
$p< s \le \frac{np}{n-p}$ which ensures that $u \in L^s_{{\rm loc}} (\R^n, dx)$ as shown above. Applying H\"{o}lder's inequality with exponents $p'$ and $p$, we obtain 
from \eqref{beta-est} and \eqref{beta-est-tr}, 
\begin{equation*}
\begin{aligned}
\int_{u \le k} |D u|^p  u^{\beta - 1} \varphi \, dx &  \le 
c \left ( \int_{B} |Du|^p u^{\beta_{0} - 1} dx \right)^{\frac{1}{p'}} \left ( \int_{B} 
u^{ \beta  p +(1-\beta_{0}) (p-1)} 
dx\right)^{\frac{1}{p}}\\& + \omega(B)^{1-\beta} \left (\int_B u \, d \omega \right)^{\beta}. 
\end{aligned}
\end{equation*}
Passing to the limit as $k \to \infty$, we deduce that \eqref{beta-0} holds if 
$\beta \le 1$ and 
$$\beta p + (1-\beta_0)(p-1) \le \beta p + p-1\le s.$$
In particular,  \eqref{beta-0} holds for $\beta=\beta_1 = \min \left(1,  \frac{s-p+1}{p}\right)$. 

If $\beta_1=1$, then $u \in W^{1, p}_{{\rm loc}}(\R^n)$ as above. In the case 
$$\beta_1 = \frac{s-p+1}{p}< 1,$$ 
we set $\beta_j=\beta_1+ \frac{p-1}{p} \beta_{j-1}$, so that 
$$
 \beta_j p +  (1-\beta_{j-1})(p-1) = s, \quad j\ge 2. 
$$
In other words, 
$$
\beta_j = \frac{s-p+1}{p} \sum_{i=0}^{j-1} \left(\frac{p-1}{p}\right)^i, \quad j=1, 2, \ldots,  
$$
Since $$\lim_{j \to \infty} \beta_j = s-p+1>1,$$
we can choose  $J\ge 2$ so that $\beta_1 \le \cdots \le \beta_{J-1}<1$, but $\beta_J \ge 1$.  

If $\beta_J>1$, then we will replace $\beta_J$ with $\beta_J=1$. Clearly, 
$$
\beta_j p+ (1-\beta_{j-1})(p-1) = s, \quad j=2, 3, \ldots,   J-1; \quad \beta_J \, p + (1-\beta_{J-1})(p-1) \le s.
$$

 Arguing by induction, and using \eqref{beta-0} with $\beta=\beta_j$, for $j=2, 3, \ldots, J$, we estimate as above,   
\begin{equation*}
\begin{aligned}
& \int_{u \le k} |D u|^p  u^{\beta_j - 1} \varphi \, dx  \le c 
\left ( \int_{B} |Du|^p u^{\beta_{j-1} - 1} dx \right)^{\frac{1}{p'}} \\ & \times \left ( \int_{B} 
u^{\beta_j p  +(1-\beta_{j-1}) (p-1)} 
dx\right)^{\frac{1}{p}} + \omega(B)^{1-\beta} \left (\int_B u \, d \omega \right)^{\beta}< \infty. 
\end{aligned}
\end{equation*}

Since $\beta_J=1$ at the last step, we arrive at the 
estimate 
$$
 \int_{u \le k} |D u|^p \varphi \, dx \le C_B <\infty,
$$
where $C_B$ does not depend on $k$. Passing to the limit as $k \to \infty$, we 
conclude that $u \in W^{1, p}_{{\rm loc}}(\R^n)$. \end{proof} 

In the next theorem we obtain a lower bound for supersolutions  of the integral equation \eqref{integral}.
\begin{theorem} \label{thmlowest}
Let $1<p<n$, $0<q<p-1$, $0<\alpha< \frac{p}{n}$, and $\sigma \in M^+(\R^n)$. Suppose 
$0\le u \in L^q_{{\rm loc}}(\R^n, d \sigma)$ is a nontrivial solution of \eqref{integraleq}.  Then the inequality 
\begin{equation}\label{lowerest}
u \ge C \, \left (\wa \sigma\right)^{\frac{p-1}{p-1-q}} \qquad d \sigma{\rm-a.e.}
\end{equation}
 holds, where $C$ is a positive constant depending only on $p, q$, and $n$. 
\end{theorem}

Before proving Theorem~\ref{thmlowest}, we recall the following lemma. 
\begin{lemma} \label{Wolfflemma} Let $1<p<\infty$ and $0<\al<\frac{n}{p}$. Then, for every $r>0$,
\begin{equation} \label{wolffest1} 
\wa\left[ (\wa\sigma)^r d\sigma \right] \ge \mathfrak{c}^{\frac{r}{p-1}} \left(\wa\sigma\right)^{\frac{r}{p-1}+1}, 
\end{equation}
where $\mathfrak{c}=\mathfrak{c}_{n, p, \al}$ depends only on $n$, $p$, and $\al$.
\end{lemma}
Estimate \eqref{wolffest1} was proved in \cite{CV1}, Lemma 3.2, 
with the constant $\frac{C_{n, p, \al}^{\frac{r}{p-1}}}{\frac{r}{p-1}+1}$ on the right-hand side. Clearly, 
$\frac{r}{p-1}+1\le e^{\frac{r}{p-1}}$, and hence, \eqref{wolffest1} follows with 
$\mathfrak{c}=e^{-1} C_{n, p, \al}$. 
\bigskip 

\begin{proof}[Proof of Theorem \ref{thmlowest}]
Let $d \omega= u^q d \sigma$. 
Fix $x \in \R^n$ and pick $R>|x|$. Let $B=B(0, R)$, and let $d\sigma_{B}=\chi_{B} \, d\sigma.$ Iterating \eqref{integraleq}, we obtain  
\begin{equation*}
\begin{aligned}
 u(x)& \ge \wa \left[(\wa \omega)^q d\sigma_{B}\right](x)\\
&= \int_0^{\infty}\left(\frac{1}{t^{n-p}} \int_{B(x, t)\cap B}\w\omega(z)^q d\sigma(z)\right)^{\frac{1}{p-1}}\frac{dt}{t}.
 \end{aligned}
 \end{equation*}

We estimate, 
$$\wa\omega(z)= \int_0^{\infty} \left(\frac{\omega(B(z, s))}{s^{n-p}} \right)^{\frac{1}{p-1}}\frac{ds}{s} \ge c \int_R^{\infty} 
\left(\frac{\omega(B(z, 2s))}{s^{n-p}} \right)^{\frac{1}{p-1}}\frac{ds}{s},$$
where $c=c(n, p, \al)>0$. 

Notice that if $z \in B$ and $R \le s $ then $B(z, 2s) \supset B(0, s).$ Hence, 
$$\wa \omega(z) \ge c \int_R^{\infty} \left(\frac{\omega(B(0, s)}{s^{n-p}} \right)^{\frac{1}{p-1}}\frac{ds}{s}.$$
From this it follows, 
$$u(x)\ge [c \, M(R)]^{\frac{q}{p-1}} \, \wa \sigma_B(x),$$
where 
$$M(R)=\int_R^{\infty} \left(\frac{\omega(B(0, s)}{s^{n-p}} \right)^{\frac{1}{p-1}}\frac{ds}{s}>0.$$
Combining \eqref{integraleq} with the preceding estimate, and using 
Lemma~\ref{Wolfflemma} with $r=q$ and 
$\sigma_B$ in place of $\sigma$, we obtain 
\begin{equation*}
\begin{aligned}
u(x)& \ge [c \, M(R)]^{(\frac{q}{p-1})^2} \, \wa \left[( \wa \sigma_B)^q d \sigma\right](x) \\
& \ge  \mathfrak{c}^{\frac{q}{p-1}}  \, [c \, M(R)]^{(\frac{q}{p-1})^2} \,\left [\wa \sigma_B(x)\right]^{1+ \frac{q}{p-1}}. 
 \end{aligned}
 \end{equation*}
Iterating this procedure and using 
Lemma~\ref{Wolfflemma} with $r= q \sum_{k=0}^{j-1} ( \frac{q}{p-1})^{k}$, we deduce 
$$ u(x)\ge \mathfrak{c}^{ \sum_{k=1}^{j} k \left (\frac{q}{p-1}\right )^k  } [c \, M(R)]^{(\frac{q}{p-1})^{j+1}}     \, 
\left[\wa\sigma_B(x)\right]^{ \sum_{k=0}^{j} (\frac{q}{p-1})^k },$$ 
for all $j=2, 3, \ldots$. Since $0<q<p-1$, obviously 
$$
\sum_{k=1}^{\infty} k \left (\frac{q}{p-1}\right)^k < \infty. 
$$
Letting $j \to \infty $ in the preceding estimate we obtain
$$ u(x)\ge C \, \left[\wa \sigma_B(x)\right]^{\frac{p-1}{p-1-q}}, \quad B=B(0, R), \quad R>|x|,$$ 
where $C>0$ depends only on $n$, $p$, $q$, and $\alpha$. Letting $R \to \infty$ yields  
\eqref{lowerest} for all $x \in \R^n$. 
\end{proof}

The next  lemma shows that if there exists a nontrivial supersolution to \eqref{integral}, 
then $\sigma$ must be absolutely continuous with respect to the $(\alpha, p)$-capacity 
defined for all $E\subset \R^n$ by  (see \cite{AH}, Sec. 2.2) 
\begin{equation*}
{\rm cap}_{\alpha, p}(E)= \inf \left \{ ||f||^p_{L^p(\R^n)} : \, \, \mathbf{I}_\alpha f \ge 1 \, \, \text{on} \, \, E, \quad f \in L^p_{+}(\R^n)\right\}.
\end{equation*}
As a consequence, if \eqref{eq1} has a nontrivial $p$-superharmonic 
supersolution, then $\sigma$ is absolutely continuous with respect to the $p$-capacity 
defined by \eqref{p-capacity}. Notice that ${\rm cap}_{p}(E)\approx {\rm cap}_{1, p}(E)$ for compact sets $E$.

\begin{lemma}\label{abscap} Let $1<p<\infty$, $0<q<p-1$,  
$0<\al < \frac{n}{p}$, and $\sigma \in \M$. Suppose there is a nontrivial solution $u \in L^q_{{\rm loc}} (\R^n, \sigma)$ to inequality \eqref{integraleq}. Then there 
exists a constant $C$ depending only on $n, p, q, \al$ such that 
\begin{equation}\label{abs-cap0}
\sigma(E)\le C\, \left[{\rm cap}_{\alpha, p}(E)\right]^{\frac{q}{p-1}}\left(\int_Eu^q d\sigma\right)^{\frac{p-1-q}{p-1}},
\end{equation}
for all compact sets 
$E \subset \mathbb{R}^n$. 
\end{lemma}
\begin{proof} Let $d \omega = u^qd\sigma$. Then $u \ge \wa \omega$ $d\sigma$-a.e. By Theorem 1.11 in \cite{V1}, 
$$\int_E \frac{d\omega}{\left(\wa \omega\right)^{p-1}} \le C \, {\rm cap}_{\alpha, p}(E),$$
where $C$ depends only on $n$, $p$, and $\al$. Hence,  
\begin{equation}\label{abs-cap1}
\int_Eu^{q-p+1}\, d\sigma \le \int_E \frac{d\omega}{\left(\wa \omega\right)^{p-1}} \le C \, {\rm cap}_{\alpha, p}(E).
\end{equation}
Note that $q-p+1<0$. Using H\"{o}lder's inequality with exponents $r=\frac{p-1}{q}$ and 
$r'= \frac{p-1}{p-1-q}$, we have
$$\sigma(E)=\int_Eu^{-\beta}u^{\beta}d\sigma \le \left (\int_Eu^{-\beta r}d\sigma\right )^{\frac{1}{r}}\left (\int_Eu^{\beta r'}d\sigma\right)^{\frac{1}{r'}},$$
where $\beta=\frac{q(p-1-q)}{p-1}>0$. Then  $-\beta r=q-p+1$ and $\beta r'=q$, and  
since $u \in L^q_{{\rm loc}} (\R^n, \sigma)$, the preceding estimate implies \eqref{abs-cap0}.
\end{proof}

\section{Solutions of the nonlinear integral equation}\label{integralequation} 

\subsection{Weighted norm inequalities and intrinsic potentials $\mathbf{K}_{\alpha, p}$.} 
Let $1<p<\infty$, $0<q<p-1$, and 
$0<\al < \frac{n}{p}$. Let $\sigma \in \M$. We denote by $\kappa$ the least constant in 
the weighted norm inequality 
\begin{equation} \label{kap-global}
||\wa\nu||_{L^q(\R^n, d\sigma)} \le \kappa  \, \nu(\R^n)^{\frac{1}{p-1}}, \quad \forall \nu \in \M.  
\end{equation}
We will also need a localized version of \eqref{kap-global} for $\sigma_E=\sigma|_E$, where $E$ is 
a Borel subset   of $\R^n$, and $\kappa(E)$ is the least constant in 
\begin{equation} \label{kap-local}
||\wa\nu||_{L^q(d\sigma_{E})} \le \kappa (E) \, \nu(\R^n)^{\frac{1}{p-1}}, \quad \forall \nu \in \M. 
\end{equation}
In applications, it will be enough to use $\kappa(E)$ 
where $E=B$ is a ball, or the intersection of two balls.

We define the intrinsic potential of Wolff type in terms of $\kappa(B(x, s))$, the least constant in \eqref{kap-local} with $E=B(x, s)$: 
\begin{equation} \label{intrinsic-K}
\mathbf{K}_{\alpha, p} \sigma (x)  =  \int_0^{\infty} \left[\frac{ \kappa(B(x, s))^{\frac{q(p-1)}{p-1-q}}}{s^{n- \al p}}\right]^{\frac{1}{p-1}}\frac{ds}{s}, \quad x \in \R^n. 
\end{equation} 

\begin{Remark}\label{varkappa1} Notice that, for $\alpha=1$,  in the definition of $\mathbf{K}_{1, p} \sigma$ we can use either the constant $\varkappa(B(x, s))$ in \eqref{weight-lap}, 
or $\kappa(B(x, s))$ in \eqref{kap-local} with $E=B(x, s)$, since by Theorem~\ref{thmpotest},  for all $E$, 
\begin{equation}\label{varkappa-kappa} 
\frac{1}{K} \, \kappa(E) \le \varkappa(E) \le K \, \kappa(E),
\end{equation}
 where $K$ is the constant in \eqref{wolff-K} which depends only on $p$ and $n$.
\end{Remark}

The proof of the following key lemma is based on Vitali's covering lemma, and weak-type maximal function inequalities. 
\begin{lemma} \label{wolffestimate} Let $1<p<\infty$, $0<q<p-1$, and 
$0<\al < \frac{n}{p}$. 

(i) Suppose $0 \le  u \in L^q_{{\rm loc}}(\R^n, d\sigma)$ is a nontrivial solution of \eqref{integraleq}. Then, for every ball 
$E=B$, \eqref{kap-local} holds with 
\begin{equation} \label{kappa-B}
\kappa(B) \le c(n, p, q, \al) \left (\int_{B} u^q d \sigma\right)^{\frac{p-1-q}{q(p-1)}}.
\end{equation}
(ii) If in statement (i) we have $u \in  L^q(\R^n, d\sigma)$, then \eqref{kap-global} 
holds with 
\begin{equation} \label{kappa-G}
\kappa \le c(n, p, q, \al) \left (\int_{\R^n} u^q d \sigma\right)^{\frac{p-1-q}{q(p-1)}}.
\end{equation}
\end{lemma}

\begin{proof}
Let $d\omega=u^q d\sigma \in M^{+}(\R^n)$. For $\nu  \in M^{+}(\R^n)$, consider the maximal function 
\begin{equation} \label{maximalfcn}
M^{\nu}_{\omega}(y)= \sup_{\rho>0}\, \left [\frac{\nu(B(y,\frac{\rho}{5}))}{\omega(B(y, \rho))}\right], \qquad y \in \R^n,
\end{equation}
where we follow the convention $\frac{0}{0}=0$. 
Let 
$$E_t= \{ y \in \R^n : \, M^{\nu}_{\omega}(y) > t \}, \quad t>0.$$
Suppose $E_t \neq \emptyset.$ Then, for every $y \in E_t$, there exists a ball $B(y, \rho_y)$ such that 
$$\frac{\nu(B(y, \frac{\rho_y}{5}))}{\omega(B(y, \rho_y))} > t.$$
Thus $E_t \subset \bigcup_{y \in E_t}B(y, \frac{\rho_y}{5})$, and hence for  any compact subset $E$ of $E_t$  there exists a $k \in \mathbb{N}$ such  that
$$E \subset \bigcup_{j=1}^{k}B \Big(y_j, \frac{\rho_{y_j}}{5}\Big).$$
Applying Vitali's covering lemma, we find  disjoint balls $\left \{B\Big(y_{j_l},\frac{\rho_{y_{j_l}}}{5}\Big)\right\}_{l=1}^{m}$ such that 
$$E \subset \bigcup_{l=1}^{m}B\Big (y_{j_l}, \rho_{y_{j_l}}\Big).$$
Consequently,
 $$\omega(E) \le \sum_{l=1}^{m}\omega\Big (B(y_{j_l}, \rho_{y_{j_l}})\Big) \le \frac{1}{t} \sum_{l=1}^{m}\nu\Big(B(y_{j_l}, \frac{\rho_{y_{j_l}}}{5})\Big) \le \frac{1}{t}\nu (\R^n). $$
Therefore, 
\begin{equation}\label{weak-type} 
\sup_{t>0} \, t \, \omega(E_t) := ||M^{\nu}_{\omega}||_{L^{1, \infty}(d\omega)} \le \nu(\R^n).
\end{equation}  
Clearly,  for any $y \in \R^n$ such that  $M_{\omega}^\nu(y)< \infty$, we have 
\begin{equation*}
\begin{aligned}
\wa\nu(y) &=\int_0^{\infty} \left(\frac{\nu(B(y, s))}{s^{n-\al p}}\right)^{\frac{1}{p-1}}\frac{ds}{s}\\& =5^{\frac{n-\al p}{p-1}}\int_0^{\infty} \left(\frac{\nu(B(y, \frac{s}{5}))}{s^{n-\al p}}\right)^{\frac{1}{p-1}}\frac{ds}{s} \\
& =5^{\frac{n-\al p}{p-1}}\int_0^{\infty} \left(\frac{\nu(B(y, \frac{s}{5}))}{\omega(B(y, s))}\cdot\frac{\omega(B(y, s))}{s^{n-\al p}}\right)^{\frac{1}{p-1}}\frac{ds}{s}\\
& \le 5^{\frac{n-\al p}{p-1}} \, \left(M_{\omega}^\nu(y)\right)^{\frac{1}{p-1}}\wa\omega(y)  \le 5^{\frac{n-\al p}{p-1}} \, \left(M_{\omega}^\nu(y)\right)^{\frac{1}{p-1}} u(y).
 \end{aligned}
 \end{equation*}

Note that if $\nu(B(y, \frac{s}{5}))>0$ but $\omega(B(y, s)) =0$ for some $s>0$ then $M_{\omega}^\nu(y)=\infty$. However, by 
\eqref{weak-type} it follows that the set of such $y\in B$ has $\omega$-measure zero, and consequently $\sigma$-measure zero, since by Lemma~\ref{average-wolff} we have $\inf_{B} u >0$.

Hence,  
\begin{equation*}
||\wa\nu||^q_{L^q(d\sigma_{B})} \le c\int_{B} \left(M_{\omega}^\nu\right)^{\frac{q}{p-1}}u^qd\sigma = c\int_{B}  \left(M_{\omega}^\nu\right)^{\frac{q}{p-1}}d\omega.
\end{equation*}

To complete our estimates, we invoke the well-known inequality
$$
||f||_{L^r(X, \omega)} \le C({r}) \,  \omega(X)^{1-r} \, ||f||_{L^{1, \infty}(X, \omega)}, 
$$
where $0<r<1$, and $\omega$ is a finite measure on $X$. Applying the preceding inequality with $r=\frac{q}{p-1}$ and $f=M_{\omega}^\nu$, we estimate
\begin{equation*}
\begin{aligned}
||\wa\nu||^q_{L^q(d\sigma_{B})} 
& \le c \,  \omega(B)^{1-\frac{q}{p-1}} \, ||M_{\omega}^\nu||^{\frac{q}{p-1}}_{L^{1, \infty}(d\omega)}\\
 & \le c  \, \omega(B)^{1-\frac{q}{p-1}} \, \nu(\R^n)^{\frac{q}{p-1}},
 \end{aligned}
 \end{equation*}
where $c$ depends only on $n, p, q, \al$. This proves statement (i) of Lemma~\ref{wolffestimate}. 

If $u \in L^q(\R^n, \sigma)$, then statement (ii) follows (i) with $B=B(x, R)$ by letting $R\to \infty$. 
\end{proof}

We will need a converse estimate to \eqref{kap-local} for subsolutions $u_B$ of equation 
\eqref{integral}  
with $\sigma_B$ in place of $\sigma$, for a ball $B$. 
\begin{Corollary}\label{kappa-uB}
Let $1<p<\infty$, $0<q<p-1$, and 
$0<\al < \frac{n}{p}$. Let $\sigma \in \M$. Suppose $u_B \in L^q(\R^n, d\sigma_B)$ is a subsolution associated  with $\sigma_B$, i.e., $0\le u_B\le \wa(u_B^q d \sigma_B)$ $d\sigma_B$-a.e. Then, for every ball 
$B$, 
\begin{equation} \label{kappa-C}
\left (\int_{B} u_B^q  d \sigma\right)^{\frac{p-1-q}{q(p-1)}} \le \kappa(B).
\end{equation}
\end{Corollary}

\begin{proof} Without loss of generality we may assume 
 $\kappa(B)<\infty$. Then using \eqref{kap-local} with 
$d \nu = u_B^q d \sigma_B$, we obtain 
$$
\int_B  u_B^q d \sigma \le \int_B 
\left[ \wa\left( u_B^q  d \sigma_B\right)\right]^q d \sigma \le \kappa(B)^q 
\left(\int_B u_B^q  d \sigma\right)^{\frac{q}{p-1}},  
$$
which yields \eqref{kappa-C}. 
\end{proof}

\subsection{Solutions in $L^q(\R^n, d\sigma)$.} The next theorem is concerned with the existence of global solutions  $u \in L^q(\R^n, d\sigma)$ to \eqref{integral}.

\begin{theorem} \label{Lqglobal}
Let $\sigma \in \M$. Then equation \eqref{integral} has a solution $u\in L^q(\R^n, d\sigma) $  if and only if there exists a constant $\kappa >0$ such that  \eqref{kap-global} 
holds. 
\end{theorem}

\begin{proof}
The necessity of \eqref{kap-global} 
follows from Lemma \ref{wolffestimate}. To prove its sufficiency, we first show that 
\eqref{kap-global} implies 
\begin{equation} \label{wolffintegrable}
\int_{\R^n} \left(\wa\sigma\right)^{\frac{q(p-1)}{p-1-q}}d\sigma < \infty.
\end{equation}

Indeed, fix a ball $B=B(x, R)$. Applying \eqref{kap-global} with $d\nu=d\sigma_B$  we obtain 
$$\int_{\R^n}(\wa\sigma_B)^qd\sigma \le \kappa^q \sigma(B)^{\frac{q}{p-1}} < \infty.$$
Letting $v_0=(\wa\sigma_B)^q$ where $v_0 \in L^1(\R^n, d\sigma)$, and 
using $d\nu=v_0 d\sigma$ in  \eqref{kap-global} we obtain
$$\int_{\R^n} \left[\wa(v_0d\sigma)\right]^qd\sigma \le \kappa^q \left(\int_{\R^n}v_0d\sigma\right)^{\frac{q}{p-1}} < \infty.$$
By Lemma~\ref{Wolfflemma} with $r=q$, we have 
\begin{equation*}
\begin{aligned}
\left[\wa(v_0d\sigma)\right]^q& =\left[\wa(\wa\sigma_B)^qd\sigma)\right]^q\\
& \ge \left[\wa(\wa\sigma_B)^qd\sigma_B)\right]^q\ge \mathfrak{c}^{\frac{q^2}{p-1}} \left(\wa\sigma_B\right)^{q(\frac{q}{p-1}+1)}.
\end{aligned}
\end{equation*}

Let $v_1= \mathfrak{c}^{\frac{q^2}{p-1}} \left(\wa\sigma_B\right)^{q(\frac{q}{p-1}+1)}$. Then $v_1 \in L^1(\R^n, d\sigma),$ and 
$$\int_{\R^n}v_1d\sigma \le \kappa^q \left(\int_{\R^n}v_0d\sigma\right)^{\frac{q}{p-1}}.$$
Applying again    \eqref{kap-global} with $d\nu=v_1d\sigma$ we obtain
\begin{equation*}
\begin{aligned}
\int_{\R^n} \left[\wa(v_1d\sigma)\right]^qd\sigma & \le \kappa^q \left(\int_{\R^n}v_1d\sigma\right)^{\frac{q}{p-1}}\\
& \le \kappa^{q(1+\frac{q}{p-1})} \left(\int_{\R^n}v_0 \, d\sigma\right)^{\frac{q^2}{(p-1)^2}} < \infty.
\end{aligned}
\end{equation*}
By Lemma~\ref{Wolfflemma} with $r=q(\frac{q}{p-1}+1)$, we estimate 
 \begin{equation*}
\begin{aligned}
 \left[\wa(v_1d\sigma)\right]^q& =\left[\wa(\mathfrak{c}^{\frac{q^2}{p-1}}(\wa\sigma_B)^{q(\frac{q}{p-1}+1)}d\sigma)\right]^q \\
& \ge \mathfrak{c}^{\frac{q^2}{p-1} (1+ 2\frac{q}{p-1}) }\left(\wa\sigma_B\right)^{q(\frac{q^2}{(p-1)^2}+ \frac{q}{p-1}+1)}.
\end{aligned}
\end{equation*}

Setting $$ v_2= \mathfrak{c}^{\frac{q^2}{p-1} (1+ 2\frac{q}{p-1}) } \left(\wa\sigma_B\right)^{q(\frac{q^2}{(p-1)^2}+ \frac{q}{p-1}+1)},$$ we obtain 
$$\int_{\R^n}v_2d\sigma \le \kappa^{q(1+\frac{q}{p-1})} \left(\int_{\R^n}v_0d\sigma\right)^{\frac{q^2}{(p-1)^2}} < \infty.$$
Arguing by induction and letting $$ v_j= \mathfrak{c}^{\frac{q^2}{p-1}\sum_{k=1}^{j} k ( \frac{q}{p-1})^{k-1}}\left(\wa\sigma_B\right)^{q \sum_{k=0}^{j} (\frac{q}{p-1})^k},$$ we obtain 
$$ \int_{\R^n} v_jd\sigma \le \kappa^{q \sum_{k=0}^{j-1} (\frac{q}{p-1})^k} \left(\int_{\R^n}v_0d\sigma\right)^{(\frac{q}{p-1})^j} < \infty.$$
By Fatou's lemma, 
$$ \int_{\R^n}\liminf_{j \to \infty}v_jd\sigma \le \liminf_{j \to \infty}\int_{\R^n}v_jd\sigma \le \kappa^{\frac{q(p-1)}{p-1-q}} < \infty.$$
Thus,  
\begin{equation} \label{wolfflocal}
\mathfrak{c}^{\frac{q^2}{p-1}\sum_{k=1}^{\infty} k ( \frac{q}{p-1})^{k-1}} \int_{\R^n} \left(\wa\sigma_B\right)^{\frac{q(p-1)}{p-1-q}}d\sigma \le \kappa^{\frac{q(p-1)}{p-1-q}} < \infty.
\end{equation}
Since $\mathfrak{c}$ and $\kappa$ in \eqref{wolfflocal}
 are independent of $B=B(x, R)$, letting $R\to \infty$ and using the Monotone Convergence Theorem  we deduce \eqref{wolffintegrable}.

Next, we let $u_0=c_0 \left(\wa \sigma \right)^{\frac{p-1}{p-1-q}}$, where $c_0 > 0$ is a small constant to be chosen later on, and construct 
a sequence $u_j$ as follows: 
\begin{equation} \label{iterations}
u_{j+1}=\wa(u^q_j d\sigma), \quad j=0,1,2, \ldots .
\end{equation}
Applying Lemma~\ref{Wolfflemma}, we estimate  
\begin{equation*}
\begin{aligned}
 u_1& =\wa(u_0^qd\sigma)=c_0^{\frac{q}{p-1}} 
 \wa \left[ \left(\wa\sigma\right)^{\frac{q(p-1)}{p-1-q}}d\sigma\right] \\& \ge c_0^{\frac{q}{p-1}} \, \mathfrak{c}^{\frac{q}{p-1-q}} 
 \left(\wa\sigma\right)^{\frac{p-1}{p-1-q}},
  \end{aligned}
 \end{equation*}
  where $\mathfrak{c}$ is the constant in \eqref{wolffestimate}. Choosing $c_0$ so that $c_0^{\frac{q}{p-1}} \, \mathfrak{c}^{\frac{q}{p-1-q}}  \ge c_0$, we obtain $u_1 \ge u_0.$ 
  
  By induction, we have  $u_j\le u_{j+1}$ ($j=0, 1, \ldots$). Note that $u_0 \in L^q(\R^n, d\sigma)$ by \eqref{wolffintegrable}. Suppose that $u_j \in L^q(\R^n, d\sigma)$, for some 
  $j \ge 0$.  Then, using \eqref{kap-global} with $d\nu=u_j^qd\sigma$, we obtain  
\begin{equation*}
\begin{aligned}
\int_{\R^n}u_{j+1}^q \, d\sigma& = \int_{\R^n}\left[\wa(u_j^q \, d\sigma)\right]^qd\sigma\\
& \le \kappa \left(\int_{\R^n}u_j^q \, d\sigma \right)^{\frac{q}{p-1}} < \infty. 
 \end{aligned}
 \end{equation*}
Since $u_j \le u_{j+1}$, the preceding inequality yields, for  all $j=0, 1, \ldots$, 
$$\int_{\R^n}u_{j+1}^qd\sigma  \le  \kappa^{\frac{p-1}{p-1-q}} < \infty.$$
Passing to the limit as $j \to\infty$ in \eqref{iterations}, we conclude using the Monotone Covergence Theorem
that $u= \lim_{j \to \infty} u_j$ is a nontrivial solution
of \eqref{integral} such that $u \in L^q(\R^n, d\sigma)$. 
\end{proof}

\subsection{Solutions in $L^q_{{\rm loc}}(\R^n, d\sigma)$.} In this subsection we 
prove the main theorem for general integral equations \eqref{integral}. We start with the following lemma. 
\begin{lemma} \label{Crx}
Suppose $0\le u \in L^q_{{\rm loc}}(\R^n, d\sigma)$ is a nontrivial solution 
of \eqref{integraleq}. Then,  for all $x \in \R^n$ and $t >0$, 
\begin{equation} \label{est1}
\sigma\left(B(x, t)\right)
\left[\int_t^{\infty} \left(\frac{ \left[\kappa(B(x, s))\right]^{\frac{q(p-1)}{p-1-q}}}{s^{n-\al p}}\right)^{\frac{1}{p-1}}\frac{ds}{s}\right]^q \le c \int_{B(x, t)} u^q d \sigma, 
\end{equation}
where $c$ depends only on $n$, $p$, $q$, and $\al$. 
\end{lemma}
\begin{proof} By Lemma~\ref{wolffestimate}, $ \kappa(B(x, s))< \infty$ for all $x \in \R^n$ and $s>0$. 
Hence it is enough to prove \eqref{necessary1} for $t$ large enough. Without loss of generality we can assume that $\sigma\not= 0$, and $\sigma(B(x, t))>0$. 
Let $d\omega=u^qd\sigma $. We estimate 
\begin{equation*} 
 \begin{aligned}
\int_{B(x, t)}u^qd\sigma & \ge \int_{B(x, t)}(\w\omega)^qd\sigma\\ & \ge \int_{B(x, t)}\left[\int_t^{\infty} \left(\frac{ \omega(B(y, s))}{s^{n-\al p}}\right)^{\frac{1}{p-1}}\frac{ds}{s}\right]^qd\sigma(y).
\end{aligned}
\end{equation*}
Since $B(y, 2s) \supset B(x, s)$ if $ s \geq t $ and $ y \in B(x, t)$, it follows, 
\begin{equation*} 
 \begin{aligned}
& \int_{B(x, t)}\left[\int_t^{\infty} \left(\frac{ \omega(B(y, s))}{s^{n-\al p}}\right)^{\frac{1}{p-1}}\frac{ds}{s}\right]^qd\sigma(y) \\& = 2^{-\frac{(n-\al p)q}{p-1}} \int_{B(x, t)}\left[\int_t^{\infty} \left(\frac{ \omega(B(y, 2s))}{s^{n-\al p}}\right)^{\frac{1}{p-1}}\frac{ds}{s}\right]^qd\sigma(y) \\
& \ge  2^{-\frac{(n-\al p)q}{p-1}} \int_{B(x, t)}\left[\int_t^{\infty} \left(\frac{ \omega(B(x, s))}{s^{n-\al p}}\right)^{\frac{1}{p-1}}\frac{ds}{s}\right]^qd\sigma(y) \\
& = 2^{-\frac{(n-\al p)q}{p-1}}\sigma(B(x, t))\left[\int_t^{\infty} \left(\frac{ \omega(B(x, s))}{s^{n-\al p}}\right)^{\frac{1}{p-1}}\frac{ds}{s}\right]^q \\
& \ge c \,  \sigma(B(x, t))\left[\int_t^{\infty} \left(\frac{\left[\kappa(B(x, s))\right]^{\frac{q(p-1)}{p-1-q}}}{s^{n-\al p}}\right)^{\frac{1}{p-1}}\frac{ds}{s}\right]^q, 
\end{aligned}
\end{equation*}
where $c=c(n, p, q, \al)$; note that in the last line we used \eqref{kappa-B}. 
Hence,  \eqref{est1} holds, which yields \eqref{necessary1} for all $x \in \R^n$ 
and $t>0$. 
\end{proof}

By picking $t$ in \eqref{est1} large enough to ensure that $\sigma(B(x,t))>0$, we deduce the following corollary. 
\begin{Corollary} \label{Cor-kap} 
Under the assumptions of Lemma~\ref{Crx}, for all $x \in \R^n$ and $t >0$,
\begin{equation} \label{necessary1}
\int_t^{\infty} \left(\frac{\left[\kappa(B(x, s))\right]^{\frac{q(p-1)}{p-1-q}}}{s^{n-\al p}}\right)^{\frac{1}{p-1}}\frac{ds}{s} < \infty.
\end{equation}
\end{Corollary}
The next lemma is an analogue of Corollary~\ref{inf} for potentials $\mathbf{K}_{\alpha, p} \sigma$.
\begin{lemma} \label{C0} Let $1<p<\infty$, $0<q<p-1$, and $0<\alpha<\frac{n}{p}$. Let 
$\sigma \in \M$. 
Suppose that \eqref{necessary1} holds for $x=0$ and $t=1$, i.e., 
\begin{equation} \label{suffcond}
\int_1^{\infty} \left(\frac{\left[\kappa(B(0,s))\right]^{\frac{q(p-1)}{p-1-q}}}{s^{n-\al p}}\right)^{\frac{1}{p-1}}\frac{ds}{s} < \infty. 
\end{equation}
Then  \eqref{necessary1} holds for all $x \in \R^n$, $t>0$, and $\mathbf{K}_{\alpha, p} 
\sigma \in L^q_{{\rm loc}}(\R^n, d \sigma)$.  
\end{lemma}
\begin{proof}
Notice that if \eqref{suffcond} holds, then obviously, for every $t>0$,  
$$\int_t^{\infty} \left(\frac{\left[\kappa(B(0, s))\right]^{\frac{q(p-1)}{p-1-q}}}{s^{n-\al p}}\right)^{\frac{1}{p-1}}\frac{ds}{s} < \infty. $$ 
For a fixed $x \in \R^n $, clearly $B(x, s) \subset B(0, s + |x|)$, so that 
\begin{equation*} 
 \begin{aligned}
 \int_t^{\infty} \left(\frac{\left[\kappa(B(x, s))\right]^{\frac{q(p-1)}{p-1-q}}}{s^{n-\al p}}\right)^{\frac{1}{p-1}}\frac{ds}{s} & \le \int_t^{\infty} \left(\frac{\left[\kappa(B(0, s+|x|))\right]^{\frac{q(p-1)}{p-1-q}}}{s^{n-\al p}}\right)^{\frac{1}{p-1}}\frac{ds}{s} \\ & =\int_{t+|x|}^{\infty} \left(\frac{\left[\kappa(B(0, r))\right]^{\frac{q(p-1)}{p-1-q}}}{(r-|x|)^{n-\al p}}\right)^{\frac{1}{p-1}}\frac{dr}{r-|x|}.
\end{aligned}
\end{equation*}
If $t>|x|$, then $r - |x| > \frac{1}{2}r$ if $r \ge t + |x|$. Hence, 
\begin{equation*} 
 \begin{aligned}
& \int_{t+|x|}^{\infty} \left(\frac{\left[\kappa(B(0, r))\right]^{\frac{q(p-1)}{p-1-q}}}{(r-|x|)^{n-\al p}}\right)^{\frac{1}{p-1}}\frac{dr}{r-|x|}\\& \le 2^{\frac{n -\al p}{ p-1}+1} \int_{t+|x|}^{\infty} \left(\frac{\left[\kappa(B(0, r))\right]^{\frac{q(p-1)}{p-1-q}}}{r^{n-\al p}}\right)^{\frac{1}{p-1}}\frac{dr}{r} < \infty.  
\end{aligned}
\end{equation*}
In the case $t \le |x|$, 
\begin{equation*} 
 \begin{aligned}
& \int_{t+|x|}^{\infty} \left(\frac{\left[\kappa(B(0, r))\right]^{\frac{q(p-1)}{p-1-q}}}{(r-|x|)^{n-\al p}}\right)^{\frac{1}{p-1}}\frac{dr}{r-|x|} \\& = \int_{t+|x|}^{2|x|} \left(\frac{\left[\kappa(B(0, r))\right]^{\frac{q(p-1)}{p-1-q}}}{(r-|x|)^{n-\al p}}\right)^{\frac{1}{p-1}}\frac{dr}{r-|x|} \\
& + \int_{2|x|}^{\infty} \left(\frac{\left[\kappa(B(0, r))\right]^{\frac{q(p-1)}{p-1-q}}}{(r-|x|)^{n-\al p}}\right)^{\frac{1}{p-1}}\frac{dr}{r-|x|}\\
& \le \kappa(B(0, 2|x|)^{\frac{q}{p-1-q}}\int_{t+|x|}^{2|x|} \left(\frac{ 1}{(r-|x|)^{n-\al p}}\right)^{\frac{1}{p-1}}\frac{dr}{r-|x|}\\
& + 2^{\frac{n -\al p}{ p-1}+1} \int_{2|x|}^{\infty} \left(\frac{\left[\kappa(B(0, r))\right]^{\frac{q(p-1)}{p-1-q}}}{r^{n-\al p}}\right)^{\frac{1}{p-1}}\frac{dr}{r} < \infty.
\end{aligned}
\end{equation*}
Combining the preceding estimates proves \eqref{necessary1} for 
all $x \in \R^n$ and $t>0$.

To show that $\mathbf{K}_{\alpha, p} 
\sigma \in L^q_{{\rm loc}}(\R^n, d \sigma)$, fix a ball $B(x, t)$ and let $B=B(x, 2t)$. 
Splitting $\mathbf{K}_{\alpha, p} \sigma$ into two parts, we estimate 
\begin{equation*} 
 \begin{aligned}
I& = \int_{B(x, t)} \left[\int_0^{t} \left(\frac{\left[\kappa(B(y, s))\right]^{\frac{q(p-1)}{p-1-q}}}{s^{n-\al p}}\right)^{\frac{1}{p-1}}\frac{ds}{s}\right]^q d \sigma(y), \\ II & =\int_{B(x, t)}  \left[\int_t^{\infty} \left(\frac{\left[\kappa(B(y, s))\right]^{\frac{q(p-1)}{p-1-q}}}{s^{n-\al p}}\right)^{\frac{1}{p-1}}\frac{ds}{s}\right]^q d \sigma(y).  
\end{aligned}
\end{equation*}
Clearly, in $II$ we have  $B(y, s) \subset B(x, 2s)$, and hence, by \eqref{necessary1}, 
$$
II \le c \sigma(B(x,t))   \left[\int_t^{\infty} \left(\frac{\left[\kappa(B(x, 2s))\right]^{\frac{q(p-1)}{p-1-q}}}{s^{n-\al p}}\right)^{\frac{1}{p-1}}\frac{ds}{s}\right]^q < \infty.
$$
In $I$, we have $B(y, s)\subset B$, and consequently, 
$$
I \le \int_B (\mathbf{K}_{\alpha, p} \sigma_B)^q d \sigma. 
$$
Since  $\kappa(B)<\infty$, by Theorem~\ref{Lqglobal} with 
$\sigma_B$ in place of $\sigma$, the equation $u_B=\wa (u_B^q d \sigma_B)$ has a solution 
$u_B$ such that $\int_B u_B^q d \sigma < \infty$. Hence by Lemma~\ref{wolffestimate} with 
$\sigma_B$ in place of $\sigma$, 
$$ \left[\kappa(B(y,s)\cap B)\right]^{\frac{q(p-1)}{p-1-q}} \le c  \int_{B(y,s)} u_B^q  d \sigma_B,$$ where $c=c(p, q, \al, n)$. From this we obtain 
$$
\int_B (\mathbf{K}_{\alpha, p} \sigma_B)^q d \sigma  \le  c^{\frac{q}{p-1}} \int_B \left[\wa (u_B^q d \sigma_B)\right]^q d \sigma=
c^{\frac{q}{p-1}} \int_B u_B^q  d \sigma <\infty.
$$
This proves  that both $I$ and $II$ are finite, i.e., $\int_B \left(\mathbf{K}_{\alpha, p} \sigma\right)^q d \sigma<\infty$. 
\end{proof}

\begin{theorem} \label{main} Let $1<p<\infty$, $0<q<p-1$, and $0<\al < \frac{n}{p}$.
Let $\sigma \in \M$. Suppose that both \eqref{wolff01} and \eqref{suffcond} hold. Then there exists a solution $u \in L^q_{{\rm loc}}(\R^n, d\sigma)$ to \eqref{integral} such that $ \liminf_{x \to \infty} u(x)=0 $, and $u$ satisfies 
the inequalities 
\begin{equation} \label{alpha}
C^{-1}   \left[\mathbf{K}_{\al, p} \sigma + \left(\wa \sigma\right)^{\frac{p-1}{p-1-q}}\right] 
\le u \le C  \left[\mathbf{K}_{\al, p} \sigma + \left(\wa \sigma\right)^{\frac{p-1}{p-1-q}}\right],  
\end{equation}
where $C>0$ is a  constant which depends only on $n$, $p$, $q$, and $\al$. 

The lower bound in \eqref{alpha} holds for any $u \in L^q_{{\rm loc}}(\R^n, d\sigma)$ 
which is a nontrivial solution of inequality \eqref{integraleq}. 
\end{theorem}

\begin{proof}
Let $u_0=c_0\left(\wa\sigma\right)^{\frac{p-1}{p-1-q}}$, where $c_0$ is a constant which will be chosen later. We construct a sequence $\{u_j\}$ as follows:   
$$u_{j+1}=\wa(u^q_j d\sigma), \quad j=0, 1, 2, \ldots .$$
Choosing $c_0$ small enough and using Lemma \ref{Wolfflemma} as in the proof of Theorem~\ref{Lqglobal}, we ensure that  $u_j \le u_{j+1}$. 

We need to verify that $u_j$ are well defined, i.e., 
$u_j \in L^q_{{\rm loc}} (\R^n, d\sigma)$.  We set $d \omega_0 = u_0^q d \sigma$. 
Let us first show that, for all $x\in \R^n$ and $t >0$, 
\begin{equation} \label{firstfinite-om}
\begin{aligned}
\omega_0(B(x, t))& =\int_{B(x, t)} u_0^q \, d \sigma \le c \,  \left[\kappa(B(x, 2t))\right]^{\frac{q(p-1)}{p-1-q}} \\ & +  c \, \left( \int_{t}^{\infty} \left(\frac{\sigma(B(x, s))}{s^{n- \al p}}\right)^{\frac{1}{p-1}} \frac{ds}{s}\right)^{\frac{q(p-1)}{p-1-q}} \, \sigma(B(x, t)),  
\end{aligned}
\end{equation}
where $c$ depends only on $n, p, q, \al$. 
We set $B=B(x, 2t)$, and denote by $B^{c}$ the complement of $B$ in $\R^n.$ Clearly, for $y \in B(x,t)$ and $0<r\le t$, $B^c \cap B(y, r)=\emptyset$. Hence, for $y \in B(x,t)$, 
 \begin{equation*}
\begin{aligned}
 \wa\sigma_{B^c}(y) & = \int_{0}^{\infty} \left(\frac{\sigma(B^c \cap B(y, r))}{r^{n- \al p}}\right)^{\frac{1}{p-1}} \frac{dr}{r}\\& = \int_{t}^{\infty} \left(\frac{\sigma(B^c \cap B(y, r))}{r^{n- \al p}}\right)^{\frac{1}{p-1}} \frac{dr}{r}.
  \end{aligned}
 \end{equation*}

If $r \ge t$, then $B(y, r) \subset B(x, 2r)$, and consequently 
\begin{equation*} 
\begin{aligned}
& \wa\sigma_{B^c}(y) \le \int_{t}^{\infty} \left(\frac{\sigma(B^c \cap B(x,2r))}{r^{n- \al p}}\right)^{\frac{1}{p-1}} \frac{dr}{r}\\
& \le \int_{t}^{\infty} \left(\frac{\sigma(B(x, 2r))}{r^{n- \al p}}\right)^{\frac{1}{p-1}} \frac{dr}{r} \le 2^{\frac{n-\al p}{p-1}} \, \int_{t}^{\infty} \left(\frac{\sigma(B(x, s))}{s^{n- \al p}}\right)^{\frac{1}{p-1}} \frac{ds}{s}. 
  \end{aligned}
\end{equation*}
From this we deduce,  
\begin{equation*}
\begin{aligned}
& \int_{B(x, t)} \left(\wa\sigma\right)^{\frac{q(p-1)}{p-1-q}}d\sigma \le c \int_{B(x, t)} 
\left(\wa\sigma_B\right)^{\frac{q(p-1)}{p-1-q}}d\sigma \\
 & + c\int_{B(x, t)} \left(\wa\sigma_{B^c}\right)^{\frac{q(p-1)}{p-1-q}}d\sigma 
 \le c \int_B \left(\wa\sigma_B\right)^{\frac{q(p-1)}{p-1-q}}d\sigma \\
 & + c \, \left( \int_{t}^{\infty} \left(\frac{\sigma(B(x, s))}{s^{n- \al p}}\right)^{\frac{1}{p-1}} \frac{ds}{s}\right)^{\frac{q(p-1)}{p-1-q}} \, \sigma(B(x, t)).
  \end{aligned}
 \end{equation*}

It follows from \eqref{suffcond} 
that $\kappa(B) < \infty$. Using Theorem \ref{Lqglobal} with $\sigma_B$ in place of $\sigma$, we see that the equation $u_B=\wa(u_B^q d\sigma_B)$ has a solution $u_B \in L^q(\R^n, d\sigma_B)$. By Theorem~\ref{thmlowest}, $u_B \ge C\left(\wa\sigma_B\right)^{\frac{p-1}{p-1-q}}$. On the other hand, by Corollary~\ref{kappa-uB}, 
$$
\int_B u_B^q \, d \sigma \le \left[\kappa(B)\right]^{\frac{q(p-1)}{p-1-q}}. 
$$
Combining  the preceding estimates proves \eqref{firstfinite-om}. In particular,  this yields $u_0 \in L^q_{{\rm loc}} (\R^n, d\sigma)$. 

We next estimate 
\begin{equation} \label{Axt}
 \begin{aligned}
A_0(x, t) & :=\int_t^{\infty} \left( \frac{\omega_0(B(x, s))}{s^{n-\al p}}\right)^{\frac{1}{p-1}}\frac{ds}{s},     
\end{aligned}
 \end{equation}
in terms of the function 
 \begin{equation} \label{Mxt}
 \begin{aligned}
M(x, t) & :=\int_t^{\infty} \left( \frac{\left[\kappa(B(x, s))\right]^{\frac{q(p-1)}{p-1-q}}}{s^{n-\al p}}\right)^{\frac{1}{p-1}}\frac{ds}{s} \\ & + \left(\int_t^{\infty} \left( \frac{\sigma(B(x, s))}{s^{n-\al p}}\right)^{\frac{1}{p-1}}\frac{ds}{s} \right)^{\frac{p-1}{p-1-q}}.     
\end{aligned}
 \end{equation}
 By Lemma~\ref{C0}, $M(x, t)< \infty$ for all $x\in \R^n$ and $t >0$. Let us show  
 that 
\begin{equation} \label{firstfinite}
A_0(x, t) \le c \, M(x,t), \qquad \text{for all} \, \,  x\in \R^n, \, \,  t >0,
\end{equation}
where $c$ depends only on $n, p, q, \al$.

Indeed,  using \eqref{firstfinite-om}, and making the substitution $\rho=2s$ in the first term, and replacing $s$ by $t\le s$ in the lower limit of integration in the second term, we 
obtain  
\begin{equation*}
 \begin{aligned}
& A_0(x, t)  \le c \int_t^{\infty} \left( \frac{\left[\kappa(B(x, 2s))\right]^{\frac{q(p-1)}{p-1-q}}}{s^{n-\al p}}\right)^{\frac{1}{p-1}}\frac{ds}{s} \\
&+ c \int_t^{\infty} \left(\frac{\sigma(B(x, s))}{s^{n-\al p}} \right)^{\frac{1}{p-1}}
 \left( \int_{s}^{\infty} \left(\frac{\sigma(B(x, \tau))}{\tau^{n- \al p}}\right)^{\frac{1}{p-1}} \frac{d\tau}{\tau}\right)^{\frac{q}{p-1-q}}   \frac{ds}{s}\\
 & \le c \int_t^{\infty} \left( \frac{\left[\kappa(B(x, \rho))\right]^{\frac{q(p-1)}{p-1-q}}}{\rho^{n-\al p}}\right)^{\frac{1}{p-1}}\frac{d\rho}{\rho} \\ &+ c \left (\int_t^{\infty} \left(\frac{\sigma(B(x, s))}{s^{n-\al p}} \right)^{\frac{1}{p-1}} \frac{ds}{s}\right)^{\frac{p-1}{p-1-q}} = c \, M(x, t).
\end{aligned}
 \end{equation*}

Now letting $d \omega_j=u^q_j \, d\sigma $, for $j=1, 2, \ldots$, we will prove the estimate 
\begin{equation} \label{firstfinite-j}
\begin{aligned}
\omega_{j}(B(x, t))& \le c \,  \left[\kappa(B(x, t))\right]^{q} \left[\omega_{j-1}(B(x, 2t))\right]^{\frac{q}{p-1}}\\ & +  c \, \left( \int_{t}^{\infty} \left(\frac{\omega_{j-1}(B(x, s))}{s^{n- \al p}}\right)^{\frac{1}{p-1}} \frac{ds}{s}\right)^{q} \, \sigma(B(x, t)),  
\end{aligned}
\end{equation}
where $c$ depends only on $n, p, q, \al$. 

We have
\begin{equation*}
\begin{aligned}
\omega_{j}(B(x, t)) & =\int_{B(x, t)} \left(\wa \omega_{j-1} \right)^qd\sigma\\& = 
\int_{B(x, t)}\left[\int_0^{\infty} \left(\frac{ \omega_{j-1}(B(y, s))}{s^{n-\al p}}\right)^{\frac{1}{p-1}}\frac{ds}{s}\right]^qd\sigma(y) \\
& \le c_q\int_{B(x, t)}\left[\int_0^t\left(\frac{ \omega_{j-1}(B(y, s))}{s^{n-\al p}}\right)^{\frac{1}{p-1}}\frac{ds}{s}\right]^qd\sigma(y) \\& + c_q\int_{B(x, t)}\left[\int_t^{\infty} \left(\frac{ \omega_{j-1}(B(y, s))}{s^{n-\al p}}\right)^{\frac{1}{p-1}}\frac{ds}{s}\right]^qd\sigma(y) \\
& :=c_q(I+II).
 \end{aligned}
 \end{equation*}

To estimate $I$, notice that if $y \in B(x, t)$ and $0<s<t$, then $B(y, s) \subset B=B(x, 2t)$. Hence, by \eqref{kap-local} with $d \nu=\chi_{B} \, d \omega_{j-1}$, we have 
\begin{equation*} 
I \le c \, \int_{B(x, t)} \left( \wa \nu \right)^q d \sigma \le 
c \, \left[\kappa(B(x, t))\right]^q \, \left[\omega_{j-1}(B(x, 2t))\right]^{\frac{q}{p-1}}.
\end{equation*}

We now estimate $II$. Since $B(y, s) \subset B(x, 2s)$ if $ y \in B(x, t)$ and $ s \ge t$, it follows that $\omega_{j-1}(B(y, s)) \le \omega_{j-1}(B(x, 2s))$ in $II$, and 
consequently 
\begin{equation*}
\begin{aligned}
II& \le  c \, \sigma(B(x, t))\left[\int_t^{\infty} \left(\frac{ \omega_{j-1}(B(x, 2s))}{s^{n-\al p}}\right)^{\frac{1}{p-1}}\frac{ds}{s}\right]^q \\
& \le c_1 \, \sigma(B(x, t))\left[\int_t^{\infty} \left(\frac{ \omega_{j-1}(B(x, s))}{s^{n-\al p}}\right)^{\frac{1}{p-1}}\frac{ds}{s}\right]^q.
\end{aligned}
 \end{equation*}
 Combining estimates $I$ and $II$, we obtain \eqref{firstfinite-j} for $j=1, 2, \ldots$.

We next estimate 
 \begin{equation}\label{A-j}
 A_j(x, t):=\int_t^{\infty} \left(\frac{ \omega_j(B(x, s))}{s^{n-\al p}}\right)^{\frac{1}{p-1}}\frac{ds}{s},
  \end{equation} 
 for $j=1, 2, \ldots$. 
Using \eqref{firstfinite-j}, and replacing the lower limit of integration $s$ 
with $t\le s$ in the second term, we estimate 
\begin{equation*}
\begin{aligned}
& A_{j}(x, t) \le c \int_t^{\infty} \left(\frac{\left[\kappa(B(x, s))\right]^{q} \left[\omega_{j-1}(B(x, 2s))\right]^{\frac{q}{p-1}}}{s^{n-\al p}}\right)^{\frac{1}{p-1}}\frac{ds}{s}\\
& + c \int_t^{\infty} \left(\frac{\left( \int_{s}^{\infty} \left(\frac{\omega_{j-1}(B(x, \tau))}{\tau^{n- \al p}}\right)^{\frac{1}{p-1}} \frac{d\tau}{\tau}\right)^{q} \, \sigma(B(x, s))}{s^{n-\al p}}\right)^{\frac{1}{p-1}}\frac{ds}{s}\\ & \le c \int_t^{\infty} \left(\frac{\left[\kappa(B(x, s))\right]^{q} \left[\omega_{j-1}(B(x, 2s))\right]^{\frac{q}{p-1}}}{s^{n-\al p}}\right)^{\frac{1}{p-1}}\frac{ds}{s} \\
& +  c \, \left[A_{j-1}(x, t)\right]^{\frac{q}{p-1}} \int_t^{\infty} \left(\frac{\sigma(B(x, s))}{s^{n-\al p}}\right)^{\frac{1}{p-1}}\frac{ds}{s}.
\end{aligned}
 \end{equation*}
Applying H\"older's inequality with exponents  $ \frac{p-1}{p-1-q}$ and $\frac{p-1}{q}$  in the first integral on the right-hand side, we obtain 
\begin{equation*}
\begin{aligned}
 A_{j}(x, t)  & \le c \, A_{j-1}(x, t)^{\frac{q}{p-1}} \left [\int_t^{\infty} \left(\frac{\left[\kappa(B(x, s))\right]^{\frac{q(p-1)}{p-1-q}}}{s^{n-\al p}}\right)^{\frac{1}{p-1}}\frac{ds}{s}\right]^{\frac{p-1-q}{p-1}} \\ & + 
 c \, \left[A_{j-1}(x, t)\right]^{\frac{q}{p-1}} \int_t^{\infty} \left(\frac{\sigma(B(x, s))}{s^{n-\al p}}\right)^{\frac{1}{p-1}}\frac{ds}{s}\\ & \le  c \, \left[A_{j-1}(x, t)\right]^{\frac{q}{p-1}} \left[M(x,t)\right]^{\frac{p-1-q}{p-1}},
\end{aligned}
 \end{equation*}
 with a different constant $c$ depending only on $n, p, q, \al$. 

Arguing by induction, we see that $A_{j}(x, t)<\infty$ for all $x \in \R^n$ and $t>0$. 
 Moreover, $ A_{j-1}(x, t) \le  A_{j}(x, t)$,  
 since $\omega_{j-1} \le \omega_j$.  Hence, from the preceding estimate we deduce 
 \begin{equation} \label{Aest}
A_j(x, t) \le C \, M(x, t), \quad   j=1, 2, \ldots, \quad {\rm for all} \, \, x \in \R^n, \, \, t>0,  
 \end{equation}
with a constant $C$ depending only on $n, p, q, \al$. An immediate consequence 
of \eqref{Aest} is the estimate 
\begin{equation*}
\omega_j(B(x, t)) \le c \,  t^{n-\al p}  \, \left[M(x, t)\right]^{p-1},  \quad   j=1, 2, \ldots, \, \, x \in \R^n, \, \, t>0,
 \end{equation*}
where $c$ depends only on $n, p, q, \al$. In particular, $u_j \in L^q_{{\rm loc}}(\R^n, \sigma)$ for all $j=0, 1, 2, \ldots$.

Thus, by the Monotone Convergence Theorem,  there exists a nontrivial solution 
to equation \eqref{integral} given by 
$$u = \lim_{j \to \infty} u_j\in L^q_{\text{loc}}(\R^n, d\sigma).$$ 
Moreover, by \eqref{Aest}, we have 
\begin{equation}\label{M-0}
\begin{aligned}
& \int_t^{\infty} \left(\frac{ \int_{B(x, s)} u^qd\sigma}{s^{n-\al p}}\right)^{\frac{1}{p-1}}\frac{ds}{s} \le C \, M(x, t) \le C \, M(x, 0),
\end{aligned}
 \end{equation} 
where the constant $C$ depends only on $n, p, q, \al$, and  $M(x, t) \to M(x, 0)$   
as $t \to 0^{+}$. Notice that 
\begin{equation*}
\begin{aligned}
M(x, 0)=\mathbf{K}_{\al, p} \sigma(x) & + \left[\wa \sigma(x)\right]^{\frac{p-1}{p-1-q}}.
\end{aligned}
\end{equation*}
Letting $t \to 0$ in \eqref{M-0} yields  
$$u(x)=\wa (u^qd\sigma)(x) =\int_0^{\infty} \left(\frac{ \int_{B(x, s)} u^qd\sigma}{s^{n-\al p}}\right)^{\frac{1}{p-1}}\frac{ds}{s}  \le C \, M(x, 0),$$
which proves the upper bound in \eqref{alpha}.  

Notice that by Lemma \ref{wolffestimate}, $\int_{B(x, s)} u^qd\sigma \ge c \, \left[\kappa(B(x, s))\right]^{\frac{q(p-1)}{p-1-q}}$. Combined with Theorem~\ref{thmlowest}, this yields the lower bound, 
$$u(x) \ge c \, M(x, 0), $$ 
for any nontrivial solution $u\in L^q_{{\rm loc}}(\R^n, d \sigma)$ 
of  $u\ge \wa (u^qd\sigma)$.  In particular, $M(x, 0) \in L^q_{{\rm loc}}(\R^n, d \sigma)$. Moreover, by Corollary \ref{inf}, we see that $\liminf_{x \to \infty} u(x)=0$. This completes the proof of Theorem~\ref{main}.
\end{proof}

\subsection{Solutions in $L^{1+q}_{{\rm loc}}(\R^n, d\sigma)$}

In this section we will prove that the solution $u\in L^q_{{\rm loc}} (\R^n, \sigma)$ to \eqref{integral} constructed in the proof of Theorem~\ref{main} 
 actually has the property $u\in L^{1+q}_{{\rm loc}} (\R^n, \sigma)$ 
under the additional assumption 
\begin{equation}\label{DB}
\int_{B} \left(\wa \sigma_B\right)^{\frac{(1+q)(p-1)}{p-1-q}}d\sigma  < \infty, \quad 
 \text{for all balls} \, \, B \, \, \text{in} \, \, \R^n.
\end{equation}  
This condition is also necessary for $u\in L^{1+q}_{{\rm loc}} (\R^n, \sigma)$.

\begin{lemma} \label{w1p} Let $1<p<\infty$, $0<q<p-1$, and $0<\al < \frac{n}{p}$.
Let $\sigma \in \M$. 
Suppose that \eqref{wolff01}, \eqref{suffcond},  
and  \eqref{DB} hold. Then 
$\wa \sigma \in L^{\frac{(1+q)(p-1)}{p-1-q}}_{{\rm loc}}(\R^n, \sigma)$, and 
$\mathbf{K}_{\al, p} \sigma \in L^{1+q}_{{\rm loc}}(\R^n, \sigma)$. 
\end{lemma}

\begin{proof} Let $x \in \R^n$ and $t>0$. We need to show  
\begin{equation*}
\begin{aligned}
I_1 &:= \int_{B(x, t)} \left(\wa \sigma\right)^{\frac{(1+q)(p-1)}{p-1-q}} \, d\sigma< \infty,\\
I_2&:=\int_{B(x,t)}\left( \mathbf{K}_{\al, p} \sigma\right)^{1+q}d\sigma< \infty.
\end{aligned}
\end{equation*}
To estimate $I_1 $, we split  $\wa \sigma$ into two integrals,  and estimate them separately,  
$$
I:=
 \int_{B(x, t)} \left[\int_0^t\left(\frac{\sigma(B(y, r))}{r^{n-\al p}}\right)^{\frac{1}{p-1}} \frac{dr}{r}\right]^{\frac{(1+q)(p-1)}{p-1-q}} d\sigma(y),$$
$$II:= 
 \int_{B(x, t)} \left[\int_t^\infty\left(\frac{\sigma(B(y, r))}{r^{n-\al p}}\right)^{\frac{1}{p-1}} \frac{dr}{r}\right]^{\frac{(1+q)(p-1)}{p-1-q}} d\sigma(y). 
$$

We first estimate $II$. 
If $r \ge t$ and $y \in B(x, t)$, then $B(y, r) \subset B(x, 2r)$, and  hence, making the substitution $s=2r$, we get 
\begin{equation*} 
II  \le c \sigma(B(x, t)) \,\left[\int_t^\infty\left(\frac{\sigma(B(x, s))}{s^{n-\al p}}\right)^{\frac{1}{p-1}} \frac{ds}{s}\right]^{\frac{(1+q)(p-1)}{p-1-q}} < \infty,
 \end{equation*}  
by Corollary~\ref{inf}, where $c=c(p, q, \al, n)$. 

To estimate $I$, notice that, if $0<r<t$ and $y \in  B(x, t)$, then $B(y, r) \subset B=B(x, 2t)$. Hence,
$$
\int_0^t\left(\frac{\sigma(B(y, r))}{r^{n-\al p}}\right)^{\frac{1}{p-1}} \frac{dr}{r}\le 
\wa \sigma_B (y), 
$$
which by \eqref{DB} yields  
\begin{equation}\label{DB-G} 
I \le \int_{B}\left(\wa\sigma_B\right)^{\frac{(1+q)(p-1)}{p-1-q}}d\sigma < \infty. 
 \end{equation} 
Thus, $I_1 < \infty$. 

We estimate $I_2$ in a similar way, splitting  $\mathbf{K}_{\al, p} \sigma$ into two integrals,  
$$
III:= \int_{B(x,t)}\left[\int_0^t \left(\frac{\left[\kappa(B(y, r))\right]^{\frac{q(p-1)}{p-1-q}}}{r^{n-\al p}}\right)^{\frac{1}{p-1}}\frac{dr}{r}\right]^{1+q}d\sigma(y),$$
$$IV:= 
 \int_{B(x,t)}\left[\int_t^\infty \left(\frac{\left[\kappa(B(y, r))\right]^{\frac{q(p-1)}{p-1-q}}}{r^{n-\al p}}\right)^{\frac{1}{p-1}}\frac{dr}{r}\right]^{1+q}d\sigma(y). 
$$

 To show that $IV<\infty$, notice that 
 $\kappa(B(y, r)) \subset \kappa(B(x, 2r))$ if $t \le r$ and $y \in B(x, t)$, which yields 
 $$
 IV \le c \sigma(B(x, t)) \left[\int_t^\infty \left(\frac{\left[\kappa(B(x, s))\right]^{\frac{q(p-1)}{p-1-q}}}{s^{n-\al p}}\right)^{\frac{1}{p-1}}\frac{ds}{s}\right]^{1+q}< \infty, 
 $$
 by Lemma~\ref{C0}, using as above the substitution $s=2r$.

 Finally, we estimate $III$. If $r<t$ and $y \in B(x, t)$, we have 
 $B(y, r) \subset B=B(x, 2t).$  
 Then $\kappa(B(y, r)) = \kappa(B(y, r)\cap B),$  and consequently, 
$$
III=\int_{B(x,t)}\left[\int_0^t \left(\frac{\left[\kappa(B(y, r)\cap B)\right]^{\frac{q(p-1)}{p-1-q}}}{r^{n-\al p}}\right)^{\frac{1}{p-1}}\frac{dr}{r}\right]^{1+q}d\sigma(y).$$

Since \eqref{DB} holds, applying Theorem 3.5  in \cite{CV1},  with $\sigma_B$ in place of $\sigma$, 
we conclude that there exists a global solution $u_B \in L^{1+q}(\R^n, d\sigma_B)$ to the equation 
$u_B=\wa(u_B^q d\sigma_B)$. By \eqref{alpha} with $\sigma_B$ in place of $\sigma$ we have  
$$
 \int_0^\infty \left(\frac{\left[\kappa(B(y, r)\cap B)\right]^{\frac{q(p-1)}{p-1-q}}}{r^{n-\al p}}\right)^{\frac{1}{p-1}}\frac{dr}{r} \le C \, u_B(y), \quad y \in \R^n.
$$

Hence,  
\begin{equation*} 
III \le  c \int_B \left(u_B\right)^{1+q} d \sigma  < \infty. 
 \end{equation*}
Thus,  both $II$ and $IV$ are finite, i.e., $I_2<\infty$. 
\end{proof}

\begin{theorem} \label{W1pmain} Suppose that \eqref{wolff01},  \eqref{suffcond}, and \eqref{DB} hold. Then there exists a nontrivial solution $u \in L^{1+q}_{{\rm loc}}(\R^n, d\sigma)$ to \eqref{integral}. Moreover, $u$ satisfies \eqref{alpha}. 

Conditions  \eqref{wolff01}, \eqref{suffcond}, and \eqref{DB} 
are necessary in order that a nontrivial solution $u \in L^{1+q}_{{\rm loc}}(\R^n, d\sigma)$ to \eqref{integral} exist. 
\end{theorem}

\begin{proof} By Theorem~\ref{main}, there exists a nontrivial solution $u \in L^{q}_{{\rm loc}}(\R^n, d\sigma)$ to the equation $u=\wa(u^qd\sigma)$ such that \eqref{alpha} holds. 
The upper estimate in \eqref{alpha} actually yields  $u \in L^{1+q}_{{\rm loc}}(\R^n, d\sigma)$ by Lemma~\ref{w1p}. 

Conditions \eqref{wolff01} and  \eqref{suffcond} are necessary for 
the existence of any nontrivial solution to \eqref{integral} by Theorem~\ref{main}. 
Condition \eqref{DB} is necessary as well which follows from \eqref{lowerest}. 
\end{proof}

\section{Proofs of Theorem~\ref{mainthm}, Theorem~\ref{W1pthm}, and Theorem~\ref{fracmain}} \label{thm}

We will need the following version of the well-known comparison principle. 
\begin{lemma} \label{cpr2} Let $\Omega$ be a bounded domain in $\R^n$. Suppose that $\mu, \nu \in L^{-1,p'}(\Omega)$, and $0 \le \mu \le \nu$. Suppose $u \in L_0^{1, p}(\Omega)$ and  $ v \in W^{1, p}(\Omega)$ are distributional solutions to the equations 
$-\Delta_p u=\mu$ and $-\Delta_p v =\nu$ in $\Omega$, respectively. Then $u \le v $ a.e. in $\Omega$.
\end{lemma}
\begin{proof} The proof is standard and relies on the use of the test function 
$\phi=u - \min\{u,v\} \in L_0^{1,p}(\Omega)$; see the proof of Lemma 3.22 in \cite{HKM}. 
\end{proof}

The next version of the comparison principle is more delicate, and we provide a detailed proof. 
\begin{lemma}\label{comparison} Suppose that $\Omega$ is a bounded open set in $\R^n$. Suppose that $\mu, \nu$ are nonnegative finite Borel measures on $\Omega$ 
such that $\mu \le \nu$, where $\mu$ is absolutely continuous with respect to the $p$-capacity ${\rm cap}_p(\cdot)$. If $u$ and $v$ are nonnegative $p$-superharmonic functions in $\Omega$ with Riesz measures $\mu$ and $\nu$, respectively, and 
$\min\{u, k\} \in L_0^{1, p}(\Omega)$ for all $k>0$, 
 then $u \le v$ a.e.
\end{lemma}

\begin{proof} Notice that $v_k = \min \, (v, k) \in W^{1, p} (\Omega)$ 
is $p$-superharmonic,  and  the corresponding Riesz measures $\nu_k = -\Delta_p v_k$ converge weakly to $\nu$ as $k \to \infty$ (\cite{HKM}, Sec. 7; \cite{KM1}). 
Let $\mu_k = \chi_{\{v<k\}} \mu$, for $k>0$. Then clearly, $\nu_k|_{\{v<k\}}=\nu|_{\{v<k\}}$, 
and consequently $\mu_k \le \nu_k$. 

For any $ \phi \in C_0^{\infty}(\Omega)$, we have 
$$\left |\int_{\Omega} \phi \, d  \mu_k - \int_{\Omega} \phi \,  d  \mu \right |= 
\left | \int_{v \ge k } \phi \, d \, \mu \right | \le \int_{v \ge k } |\phi| \, d \, \mu  \le \max_{\Omega} \, |\phi| \, \mu\left(\{v \ge k \}\right).$$

We have $\mu\left(\{v \ge k \}\right) \to \mu\left(\{v = \infty \}\right)$ as $k \to \infty$.
Since $\mu$ is absolutely continuous with respect to $\text{cap}_{p}(\cdot)$, 
and $v$ is $p$-superharmonic, it follows that  $\mu\left(\{v = \infty \}\right)=0$, which yields $\mu_k \to \mu$ weakly.

Let us denote by $u_k$ the unique solution to the equation 
$$-\Delta_p u_k = \mu_k, \quad u_k \in L_0^{1, p}(\Omega),$$ where $\mu_k \in L^{-1, p'}(\Omega)$ since $\mu_k \le \mu \in L^{-1, p'}(\Omega)$. 
 By Lemma~\ref{cpr2}, we have $u_k \le v_k$ for every $k>0$, and $u_k \le u_j$ if $k \le j$. Passing to the limit as 
$k \to \infty$, we obtain $\tilde u \le v$, where $\tilde u = \lim_{k \to \infty} \, u_k$. 
Since 
$\mu_k \to \mu$ weakly, it follows that $\tilde u$ is a $p$-superharmonic 
solution to the equation $-\Delta_p \tilde u=\mu$ where $\min \, (\tilde u, j) \in 
L_0^{1, p}(\Omega)$ for every $j>0$. Since $\mu$ is absolutely continuous with respect to the $p$-capacity 
${\rm cap}_p(\cdot)$, and $\min\, (u, k) \in L_0^{1, p}(\Omega)$ for every $k>0$, it follows 
by the uniqueness theorem (see 
\cite{K2}, and the references given there) that $\tilde u = u$ a.e., and consequently   
$u \le v$ a.e.  
\end{proof}

\begin{proof}[Proof of Theorem~\ref{mainthm}] Let $1<p<n$. Suppose both \eqref{Wolfinite} 
and \eqref{suffcond1} hold. Then by Theorem~\ref{main} there exists a nontrivial solution $v \in L^q_{{\rm loc}}(d\sigma)$ of the equation
\begin{equation}\label{V}
v = K \, \w(v^qd\sigma) \quad \text{in} \, \, \R^n, 
\end{equation}
where $K$ is the constant in Theorem~\ref{thmpotest}. By  Theorem~\ref{thmlowest} (with $K^{p-1} \sigma$ in place of $\sigma$),  
$$
v \ge C  \, K^{\frac{p-1}{p-1-q}} (\w \sigma)^{\frac{p-1}{p-1-q}}, 
$$
where $C$ is the constant in \eqref{lowerest}. We set 
$$w_0= c_0 \left(\w\sigma\right)^{\frac{p-1}{p-1-q}}, \quad d\omega_0=w_0^q \, d\sigma, $$ 
where $c_0>0$ is a small constant to be determined later. 
In particular, we pick $c_0\le C \, K^{\frac{p-1}{p-1-q}}$ so that 
 $$w_0 \le \frac {c_0}{C \, K^{\frac{p-1}{p-1-q}}} \, v \le v.$$  

Clearly $\omega_0$ is a locally finite Borel measure since $d\omega_0 \le v^qd\sigma$ and $v \in L^q_{\text{loc}}(d\sigma)$. By Lemma~\ref{abscap} with $\alpha=1$, 
$\omega_0$ is absolutely continuous with respect to ${\rm cap}_p(\cdot)$. Hence there exists a unique renormalized solution 
(see \cite{K2}) to the equation
\begin{equation}\label{eq-u1}
 -\Delta_p u_1^k ={\omega_0} \, \chi_{B(0, 2^k)} \, \,  \text {in $B(0, 2^k)$}, \quad u_1^k =0 \, \text{ on } \, \partial B(0, 2^k), 
 \end{equation}
 where $k=0, 1, 2, \ldots$.  
Notice that the sequence $\{u_1^k\}$ is increasing by the comparison principle (Lemma~\ref{comparison}).
Moreover, by Theorem~\ref{thmpotest}, 
$$0 \le u_1^k \le K \, \w (\omega_0\chi_{B(0,2^k)}) \le K \, \w \omega_0 \le K \, \w (v^q d\sigma) =  v.$$

Letting $u_1=\lim_{k \to \infty} u_1^k $ and using the weak continuity of the $p$-Laplace operator (Theorem~\ref{weakconv}) and the Monotone Convergence Theorem, we see that $u_1$ is a 
$p$-superharmonic solution to the equation $-\Delta_p u_1 = \omega_0$ in $\R^n$. 
Since $u_1^k \le v$, it follows that $u_1 \le v$, and hence $\liminf_{|x| \to \infty} u_1(x)=0$. By Theorem~\ref{thmpotest}, 
$$0 \le u_1 \le K \, \w \omega_0 \le K \, \w (v^q d\sigma) =  v.$$
We deduce, using \eqref{wolffest1}, 
\begin{equation*}
\begin{aligned}
u_1 & \ge \frac{1}{K}   \w \omega_0=\frac{c_0^{\frac{q}{p-1}}}{K} 
\,  \w \Big[ \left(\w \sigma\right)^{\frac{q(p-1)}{p-1-q}} d \sigma\Big] \\
\\ & \ge \frac{c_0^{\frac{q}{p-1}} \mathfrak{c}^{\frac{q}{p-1-q}}}{K} 
\,  \left(\w \sigma\right)^{\frac{p-1}{p-1-q}} =   \frac{c_0^{\frac{q}{p-1}-1} \mathfrak{c^{\frac{q}{p-1-q}}}}{K} \, w_0.
\end{aligned}
\end{equation*}
Hence, for $c_0\le \min \left[ \left (\mathfrak{c}^{\frac{q}{p-1-q}} \, K^{-1}\right)^{\frac{p-1}{p-1-q}}, C \, K^{\frac{p-1}{p-1-q}}\right]$, we have $v \ge u_1\ge w_0$.

Let us now construct a sequence $u_j$ ($j=1, 2, \ldots$) of functions which are $p$-superharmonic in $\R^n$, $u_j \in L^q_{\text{loc}}(d\sigma)$, 
so that 
\begin{equation}\label{sequenceuj}
\left\{ \begin{array}{ll}
-\Delta_pu_{j}=\sigma u_{j-1}^q \, \, \text{ in ~~} \mathbb{R}^n, \quad j=2, 3, \ldots, \\
c_j \, \left(\w\sigma\right)^{\frac{p-1}{p-1-q}}\le u_{j} \le v,\\
0\le u_{j-1} \le u_j , \\
\liminf_{|x| \to \infty} u_j (x)=0.
\end{array} \right.
\end{equation}
Here $c_1=c_0^{\frac{q}{p-1}} \mathfrak{c}^{\frac{q}{p-1-q}} \, K^{-1}$, and 
$$c_j=\left(\mathfrak{c}^{\frac{q}{p-1-q}} \, K^{-1}\right)^{\sum_{l=0}^{j-1}\left(\frac{q}{p-1}\right)^l} c_0^{ \left(\frac{q}{p-1}\right)^{j}},  \quad j=2, 3, \ldots.$$

Suppose that $u_1, \ldots, u_{j-1}$ have been constructed. Let $d\omega_{j-1}= u^q_{j-1}d\sigma$. Then $\omega_{j-1}\in \M$,  since $u_{j-1} \le v$, where  $v \in L^q_{\text{loc}}(d\sigma)$, and $\omega_{j-1}$ is absolutely continuous with respect to the $p$-capacity. Applying Lemma~\ref{comparison} again, we see that there exists a renormalized solution $u_j^k$ to the equation
$$ -\Delta_p u_{j}^k ={\omega_{j-1}}\chi_{B(0, 2^k)} \, \,  \text { in $B(0, 2^k)$}, \quad u_j^k =0 \, \, \text{ on }\,\,  \partial B(0, 2^k).  $$
Arguing by induction, let  $u^k_{j-1}$ be the unique solution of the equation
$$ -\Delta_p u_{j-1}^k ={\omega_{j-2}}\chi_{B(0, 2^k)} \, \,  \text { in $B(0,2 ^k)$}, \quad u_{j-1}^k =0 \, \, \text{ on } \partial B(0,2^k).  $$
Since $u_{j-2} \le u_{j-1}$, by Lemma~\ref{comparison}, we deduce $u_j^k \ge u^k_{j-1}.$ Using Theorem~\ref{thmpotest}, we have
$$0 \le u_j^k \le K \, \w \left[\omega_{j-1}\chi_{B(0, 2^k)}\right]  \le K \, \w (v^q d\sigma) =  v.$$
Letting $u_{j}=\lim_{k \to \infty}u_j^k$ and using again the weak continuity of the $p$-Laplacian and the Monotone Convergence Theorem, we deduce that $u_j$ is a solution to the equation $-\Delta_p u_j = \sigma u^q_{j-1}$ on $\R^n$. 

Moreover, $u_j \le v$ since $u_j^k \le v$ and hence $\liminf_{x\to \infty} u_j(x)=0.$ Furthermore, we have $u_{j-1} \le u_j$ since $u_{j-1}^k \le u_j^k$, for all $k \ge 1.$
On the other hand, applying Theorem \ref{thmpotest} and Lemma \ref{Wolfflemma}, and arguing by induction, we obtain 
\begin{equation*}
\begin{aligned}
u_j & \ge \frac{1}{K} \w(u^q_{j-1}d\sigma)\ge \frac{1}{K} \w\left[c_{j-1}^q \, \left(\w\sigma\right)^{\frac{q(p-1)}{p-1-q}} d \sigma\right] \\
& \ge \mathfrak{c}^{\frac{q}{p-1-q}} c^{\frac{q}{p-1}}_{j-1}  \, K^{-1}
 \left(\w\sigma\right)^{\frac{p-1}{p-1-q}}= c_j\, \left(\w\sigma\right)^{\frac{p-1}{p-1-q}}.
\end{aligned}
\end{equation*}

Letting $u=\lim_{j \to \infty}u_j$ and using Theorem~\ref{weakconv} together with  the Monotone Convergence Theorem, we see that 
$u$ is a solution to the equation  $-\Delta_p u = \sigma u^q$ on $\R^n$. Hence, by Theorem~\ref{thmpotest}, $u \ge \frac{1}{K} \w\left(u^q d \sigma\right)$. Applying 
Theorem~\ref{main}, we deduce the lower bound in \eqref{two-sided}. The upper bound 
follows from $u \le v$ and Theorem~\ref{main}. We also have $\liminf_{x \to \infty}u(x)=0$ since $ u \le v$, and $\liminf_{x\to \infty} v(x)=0$ by Corollary~\ref{inf}.
Notice that by Remark~\ref{varkappa1} we can use here the potentials $\mathbf{K}_{1, p} \sigma$ defined either in terms of $\varkappa(B)$, or $\kappa(B)$ in the case $\alpha=1$, 
 since they are equivalent.

Let us now prove the minimality of $u$. Suppose $w \in L^q_{\text{loc}}(\R^n, d\sigma)$ is any nontrivial $p$-superharmonic solution to \eqref{eq1}. Let $ d\nu =w^qd\sigma$. 
Then by 
Theorem~\ref{thmpotest}, $w \ge \frac{1}{K} \w (w^qd\sigma)$. 
Hence, by Lemma~\ref{abscap} with $\alpha=1$, $\nu$ is absolutely continuous with respect to the $p$-capacity. By Theorem~\ref{thmlowest} with $K^{1-p} \, \sigma$ in place 
of $\sigma$,
$$w \ge C \, K^{\frac{1-p}{p-1-q}} \, \left(\w\sigma\right)^{\frac{p-1}{p-1-q}}.$$
Note that by the choice of $c_0$ above, we have  $\omega_0 \le \nu$. Therefore, 
by 
Lemma~\ref{comparison}, 
the function $u_1^k$ defined by \eqref{eq-u1} satisfies the inequality 
$u_1^k \le w$ in $B(0, 2^k)$ for every $k>0$, and consequently   $u_1 = \lim_{k \to \infty} u_1^k \le w.$ 
Repeating this argument by induction, we obtain $u_j \le w$ for every $j=1, 2, \ldots$. It follows that $\lim_{j \to \infty}u_j = u \le w$, which proves the minimality of $u$. 
This completes the proof of statement (i) of Theorem~\ref{mainthm}. 

To prove statement (ii), suppose that $u$ is a supersolution of \eqref{eq1}. Then by 
Theorem~\ref{thmpotest}, $u \ge \frac{1}{K} \w\left(u^q d \sigma\right)$. Hence, by 
Theorem~\ref{main}, both \eqref{Wolfinite} 
and \eqref{suffcond1} hold.

Statement (iii) is an immediate consequence of Theorem~\ref{thmpotest} (ii). 
\end{proof}

We now are in a position to give a characterization of  $W^{1, p}_{{\rm loc}}$-solutions of \eqref{eq1} stated in Theorem~\ref{W1pthm}. 
We remark that a global analogue of condition \eqref{locWolff},  as was shown earlier by the authors \cite{CV1}, 
is necessary and sufficient for the existence of a finite energy solution $u\in L^{1, p}_0(\R^n)$ to \eqref{eq1}.

\begin{proof}[Proof of  Theorem \ref{W1pthm}]   By Theorem~\ref{mainthm}, if both \eqref{Wolfinite} and \eqref{suffcond1} hold, then  
there exists a $p$-superharmonic solution $u$ to \eqref{eq1} such that \eqref{two-sided} holds. Moreover, by \eqref{Uest}, there is a constant $K>0$ such that 
\begin{equation}\label{lowest-A}
u\ge \frac{1}{K} \w (u^q \, d \sigma). 
\end{equation}

Suppose that additionally \eqref{locWolff} holds for all balls $B$. Applying Theorem~\ref{W1pmain}, we see that there exists a solution $v \in L_{\text{loc}}^{1+q}(\R^n, d\sigma)$ to the integral equation \eqref{integral} 
such that \eqref{alpha} holds with $\alpha=1$. Hence, there exists a constant $c>0$ such that 
$$
c^{-1} \, u(x) \le v(x) \le c \, u(x) \quad d \sigma-\text{a.e.}
$$
Consequently, $u \in L_{\text{loc}}^{1+q}(\R^n, d\sigma)$, and 
$$
\int_B \w (u^q d \sigma_B) \, u^q \, d \sigma \le C \int_B u^{1+q} d \sigma < \infty,
$$ 
for every ball $B$. By a local version of Wolff's inequality \eqref{wolff-ineq-loc}, we see that $u^q d \sigma \in W^{-1, p'}_{{\rm loc}} (\R^n)$. 
Applying Lemma~\ref{loc-sol}, we conclude that $u \in W^{1, p}_{{\rm loc}} (\R^n)$. 

Conversely, if there exists a nontrivial solution $u \in W^{1, p}_{{\rm loc}} (\R^n)$ to \eqref{eq1}, then clearly a quasi-continuous representative of $u$  is a $p$-superharmonic solution, and $u^q d \sigma \in W^{-1, p'}_{{\rm loc}} (\R^n)$. It follows from 
\eqref{Uest} that 
$$
u \le K \, \w (u^q \, d \sigma). 
$$

By Wolff's inequality \eqref{wolff-ineq}, for every ball $B$, 
$$\int_B u^{1+q} \, d \sigma \le K \int_B \w (u^q d \sigma_B) \, u^q \, d \sigma_B 
\le C \, K \, ||u^q \, d \sigma_B ||_{L^{-1, p'} (\R^n)} < \infty.  
$$
By Theorem~\ref{thmlowest}, estimate \eqref{lowerest} holds. Combining these estimates, 
we obtain that \eqref{locWolff} holds for all balls $B$. 
By Theorem~\ref{mainthm}, both \eqref{Wolfinite} and \eqref{suffcond1} hold as well, 
which completes the proof of Theorem~\ref{W1pthm}. \end{proof}

\begin{proof}[Proof of Theorem \ref{fracmain}]
We remark that \eqref{fraclap} is understood in the sense 
$$u = {\bf I}_{2 \al} (u^q \, d \sigma) \quad  \, \, \text{in} \, \, \R^n, \quad u \ge 0.$$
Since ${\bf I}_{2 \al} (u^q \, d \sigma) = {\bf W}_{\al, 2} (u^q \, d \sigma)$, 
 Theorem~\ref{fracmain} is a special case of Theorem~\ref{main} with $p=2$. \end{proof}

\begin{Remark}
{\rm \textbf{(1)} Direct analogues of our main theorems hold for the more general quasilinear 
$\mathcal{A}$-Laplace 
operator $\text{div} \,  \mathcal{A}(x, \nabla u)$ in place of $\Delta_p$:
  \begin{equation}\label{A-p}
  - \text{div} \,  \mathcal{A}(x, \nabla u)= \sigma u^q \quad {\rm in} \, \, \R^n, \qquad \liminf_{x \to \infty} u=0, 
    \end{equation}
  under the standard monotonicity and boundedness assumptions on $\mathcal{A}$ which 
guarantee that the Wolff potential estimates \eqref{Uest} hold (see, e.g., \cite{KM2}, \cite{KuMi}, 
\cite{TW1}, \cite{PV1}).  

\textbf{(2)} Similar results hold for the fully nonlinear $k$-Hessian operator $F_k$  ($k=1, 2, \ldots, n$)  defined by
 \begin{equation}\label{I01}
 F_k[u] =\sum_{1\leq i_{1}<\cdots<i_{k}\leq n}\lambda_{i_{1}}\cdots\lambda_{i_{k}},
  \end{equation}
where $ \lambda_{1},\ldots,\lambda_{n}$ are the eigenvalues of the
Hessian matrix $D^{2}u$ on $\R^n$. In other words, $F_{k}[u]$ is the sum of
the $k\times k$ principal minors of $D^{2}u$, which coincides with
the Laplacian $F_1[u] = \Delta u$ if $k=1$. 

Local Wolff potential estimates for the equation $F_k[u]=\mu$, where $\mu \in \M$, in this case are due to Labutin \cite{Lab} (see also \cite{TW1}); global estimates analogous to \eqref{Uest} can be found in \cite{PV1}. The corresponding ``sublinear'' equation can be written in  the form 
\begin{equation}\label{F-k} 
F_k[u] = \sigma \, |u|^q \quad {\rm in} \, \, \R^n, \qquad \limsup_{x \to \infty} \, u=0, 
\end{equation}
where $0<q<k$, and $u\le 0$ is a $k$-convex function. 

Similar equations in the supercritical  case $q>k$ were considered in \cite{PV1}, and in the critical case $q=k$, in \cite{JV1}. Intrinsic nonlinear potentials of the type 
$\mathbf{K}_{\alpha, p} \sigma$ do not play a role there. However, the reduction of both 
\eqref{A-p} and 
\eqref{F-k} to \eqref{integral} 
is carried over as in the case of the $p$-Laplacian treated above. See details in \cite{PV1}, \cite{JV1}, \cite{JV2}, \cite{CV1}. }
\end{Remark}

\section{Example}\label{eg}

Suppose  $0<q<1$, $n \ge 2$, and  $0< \al <\frac{n}{2}$. In this section we construct 
$\sigma \in M^+(\R^n)$ such that $\kappa(B(0, R)) < \infty$ for every $R>0$, and 
the equation 
$$\left\{ \begin{array}{ll}
(-\Delta)^{\al} u \, =  \sigma & \text{in ~} {\R}^n,\\ 
\displaystyle{\liminf_{x \to \infty} u(x) =0}, & \, 
\end{array} \right.$$ 
has a weak solution, but the equation 
$$\left\{ \begin{array}{ll}
(-\Delta)^{\al} u \, =\sigma \, u^q & \text{in ~} {\R}^n,\\ 
\displaystyle{\liminf_{x \to \infty} u(x) =0}, & \, 
\end{array} \right.$$ 
has no weak solutions. The condition $\kappa(B(0, R)) < \infty$ ensures that locally, 
for $\sigma_{B(0, R)}$ in place of $\sigma$, weak solutions exist.

In other words, we need to construct a measure $\sigma$ such that $\mathbf{I}_{2 \al} \sigma < \infty$ 
a.e., that is, 
\begin{equation}\label{6a}
\int_1^\infty \frac {\sigma(B(0,R))}{R^{n-2 \al}} \frac{dR}{R} < \infty,  
\end{equation}
and $\kappa(B(0, R)) < \infty$ for every $R>0$, 
but 
\begin{equation}\label{6b}
\int_1^\infty \frac{\left[\kappa(B(0, R))\right]^{\frac{q}{1-q}}}{R^{n-2 \al}}\frac{dR}{R} = \infty.   
\end{equation}
 This requires $\kappa(B(0, R))^{\frac{q}{1-q}}$ to grow much faster than 
$\sigma(B(0, R))$ as $R \to \infty$.

\begin{lemma}\label{lemma6-1} Let $0<q<1$ and $0<2 \al < n$. If   
\begin{equation}\label{ineq}
|| \mathbf{I}_{2 \al} \nu||_{L^q(d \sigma)} \le \kappa(\sigma) \,  \nu(\R^n), \quad \forall \nu \in \M, 
\end{equation}
then 
\begin{equation}\label{cond}
\mathcal{K}(\sigma)\!: =\sup_{x \in \R^n} \int_{\R^n} \frac{d \sigma (y)}{|x-y|^{(n-2 \al)q}} \le \kappa(\sigma)^q. 
\end{equation}
\end{lemma}
\begin{proof} Let $\nu = \delta_x$ in  (\ref{ineq}), and take the supremum of the left-hand side over all $x \in \R^n$. 
\end{proof} 

We will need the following lemma in the radially symmetric case which will be proved elsewhere. 
\begin{lemma}\label{lemvar} Let $0<q<1$ and $0<2 \al < n$. If   $d \sigma = \sigma(|x|) \, dx$ is radially symmetric 
then condition (\ref{cond}) is equivalent to $\mathbf{I}_{2 \al} \sigma \in L^{\frac{1}{1-q}, \frac{q}{1-q}} (d \sigma)$, 
and hence is not only necessary, but also sufficient for (\ref{ineq}). Moreover,  there exists  $c=c(q, \al, n)>0$ 
such that the least constant $\kappa(\sigma)$ in \eqref{ineq} satisfies 
\begin{equation}\label{const}
\mathcal{K}(\sigma) \le \kappa(\sigma)^q \le c \, \mathcal{K}(\sigma).  
\end{equation}
\end{lemma}

\begin{cor}\label{cor-c} Let $\sigma_{R, \gamma} = \chi_{B(0, R)} |x|^{-\gamma}$, where 
$0\le \gamma<n-q(n-2 \al)$ and $R>0$. Then 
\begin{equation}\label{rad-const}
\frac{\omega_n}{n-\gamma - q(n-2 \al)}     \le \frac{\kappa(\sigma_{R, \gamma})^q} {R^{{n-\gamma-q(n-2 \al)}}}
 \le \frac{c}{n-\gamma - q(n-2 \al)},  
\end{equation}
where $c=c(q, \al, n)$, and $\omega_n= |S^{n-1}|$ is the surface area of the unit sphere. 
\end{cor}
\begin{proof} Letting $x=0$ in (\ref{cond}) we have 
\begin{equation*}
\begin{aligned}
\mathcal{K}(\sigma_{R, \gamma}) & = \int_{|y|<R} \frac{|y|^{-\gamma}}{|y|^{q(n-2 \al)}} dy= 
\omega_n  \int_0^R r^{-\gamma -q(n-2 \al)+n-1} dr \\
&=\frac{\omega_n}{n-\gamma - q(n-2 \al)} R^{n-\gamma-q(n-2 \al)}. 
\end{aligned}
\end{equation*}
Hence, (\ref{const}) follows from the preceding estimate and Lemma~\ref{lemvar}. 
\end{proof}

Let 
$$
\sigma = \sum_{k=1}^\infty c_k \sigma_{k, \gamma_k} (x+x_k), 
$$
where $|x_k|=k$, $\gamma_k = n-q(n-2 \al) - \epsilon_k$, and $c_k$, $\epsilon_k$ 
are picked so that $\sum_{k=1}^\infty c_k < \infty$, and $\epsilon_k \to 0$ fast enough; 
it suffices to set $$c_k=\frac{1}{k^{2}},  \quad \epsilon_k = \frac{1}{k^{n+2}}.$$

Let $R>0$. Clearly, 
$$
\sigma(B(0, R)) \le  \sum_{k=1}^\infty c_k  \sigma_{k, \gamma_k} (B(x_k, R)) 
\le \sum_{k=1}^\infty c_k \sigma_{k, \gamma_k} (B(0, R)). 
$$
Here 
\begin{equation*}
\begin{aligned}
 \sigma_{k, \gamma_k} (B(0, R)) & = \omega_n \int_0^{\min (k, R)} r^{-\gamma_k + n-1} dr  \\
&= \frac{\omega_n}{n-\gamma_k} \min (k, R)^{n-\gamma_k}\le  \frac{\omega_n}{q(n-2 \al)} 
\min(k, R)^{q(n-2 \al) + \epsilon_k}. 
\end{aligned}
\end{equation*}
Hence, for $R\ge 1$ 
$$
\sigma (B(0, R)) \le \frac{\omega_n}{q(n-2 \al)} \sum_{k=1}^N c_k k^{q(n-2 \al) + \epsilon_k} 
+ \frac{\omega_n}{q(n-2 \al)}  R^{q(n-2 \al) +\epsilon_N}  \sum_{k=N}^\infty  c_k.  
$$ 
Picking $N$ large enough so that $\epsilon_N < (1-q)(n-2 \al)$, we obtain \eqref{6a}.

Using Corollary \ref{cor-c}, we will show that  $\kappa(B(0, R)) < \infty$ for every $R>0$, 
since $\epsilon_k>0$, and consequently $\gamma_k$ is below the critical exponent $n-q(n-2 \al)$.
Indeed, since $\kappa(\sigma)$ is obviously invariant under translations, 
\begin{equation}\label{6-sum}
[\kappa(B(0, R))]^q \le  \sum_{k=1}^\infty c_k  \left[\kappa( \chi_{B(x_k,R)}\sigma_{k, \gamma_k})\right]^q. 
\end{equation}
If $k>2R$, then $|x-x_k|< R$, $|x|<k$ and $|x_k|=k$ yield $k>|x|>\frac k 2$.  Consequently, 
$ \chi_{B(x_k,R)} \sigma_{k, \gamma_k}(x) \approx  \frac{c}{k^{\gamma_k}} \chi_{B(x_k,R)}$. 
It follows that, for $\nu \in \M$,  
\begin{equation*}
\begin{aligned}
|| \mathbf{I}_{2 \al} \nu||^q_{L^q(\chi_{B(x_k,R)}d \sigma_{k, \gamma_k})} & \le \frac{c}{k^{\gamma_k}} 
|| \mathbf{I}_{2 \al} \nu||^q_{L^q(\chi_{B(x_k,R)}dx)}\\ &\le \frac{c}{k^{\gamma_k}}  \left[\kappa( \chi_{B(x_k,R)})\right]^q \nu(\R^n)^q.
\end{aligned}
\end{equation*}
Corollary \ref{cor-c} with $\gamma=0$, yields $\left[\kappa( \chi_{B(x_k,R)})\right]^q\approx R^{n-q(n-2 \al)}$. Hence, 
$$
\left[\kappa( \chi_{B(x_k,R)}\sigma_{k, \gamma_k})\right]^q \le \frac{c}{k^{\gamma_k}} R^{n-q(n-2 \al)}. 
$$ 
From this and \eqref{6-sum} we deduce 
$$
[\kappa(B(0, R))]^q \le  \sum_{1\le k \le 2R} c_k  [\kappa(\sigma_{k, \gamma_k})]^q 
+ c R^{n-q(n-2 \al)} \sum_{k > 2R}^\infty \frac{c_k} {k^{\gamma_k}} < \infty.   
$$
Note that each term in the first  sum is finite by Corollary \ref{cor-c}
since $0<\gamma_k<n-q(n-2 \al)$ is below the critical 
exponent.

Let us show now that \eqref{6b} holds. By Lemma \ref{lemma6-1}, 
\begin{equation*}
\begin{aligned}
&[ \kappa(B(0,R))]^q \ge \mathcal{K}(\sigma_{B(0, R)}) =\sup_{x \in \R^n}       \sum_{k=1}^\infty c_k  \int_{|y|<R} \frac {\sigma_{k, \gamma_k}(y+x_k)} {|x-y|^{q(n-2 \al)}} dy \\
& \ge \sup_{k\ge 1} c_k  \int_{|y|<R} \frac {\sigma_{k, \gamma_k}(y+x_k)} {|x_k+y|^{q(n-2 \al)}} dy 
= \sup_{k\ge 1}  c_k  \int_{|z-x_k|<R} \frac {\sigma_{k, \gamma_k}(z)} {|z|^{q(n-2 \al)}} dz 
\\
& =\sup_{k\ge 1}  c_k  \int_{|z-x_k|<R, \, |z|< k} \frac {dz} {|z|^{\gamma_k + q(n-2 \al)}} 
= \sup_{k\ge 1}  c_k  \int_{|z-x_k|<R, \, |z|< k} \frac {dz} {|z|^{n-\epsilon_k}}.  
\end{aligned}
\end{equation*}
If  $k\le \frac R 2$, then $B(0, k) \subset B(x_k, R)$. Hence,  for $R >2$, 
\begin{equation*}
\begin{aligned}
& [ \kappa(B(0,R))]^q \ge  \sup_{1\le k\le \frac{R}{2}}  c_k  \int_{|z|< k} \frac {dz} {|z|^{n-\epsilon_k}}  
\ge 
\omega_n \sup_{1 \le k\le \frac R 2}  \frac{c_k}{\epsilon_k} \, k^{\epsilon_k}\\
& \ge \omega_n 
 \sup_{\frac R 4 \le k\le \frac R 2}  \frac{c_k}{\epsilon_k} \ge  \omega_n 4^{-n} R^n.  
\end{aligned}
\end{equation*}
Since $\frac{n}{1-q}> n - 2 \al$, the preceding estimate yields \eqref{6b} as claimed. 

\section{Equations with singular gradient terms} \label{gradeq}

In this section, we 
 investigate the relationship between \eqref{eq1} and \eqref{ric-eq}, and 
 prove Theorem~\ref{main-riccati}  using the framework of (locally) renormalized solutions. We will show that transformation \eqref{substitution} sends a solution $u$ of \eqref{eq1} to a solution $v$ of \eqref{ric-eq}, but in the opposite direction, a solution $v$ of \eqref{ric-eq}  generally gives rise merely to a supersolution $u$ of \eqref{eq1}. Note that $u$ is a genuine solution only  under some 
 additional assumptions on $v$ as is clear from the following example.
 
  For  $0<q<1$, $p=2$,  $n \ge 3$, and $\sigma=0$, obviously, $v=c |x|^{(1-q)(2-n)}$ is a weak solution of \eqref{ric-eq} for an appropriate $c>0$, but the corresponding $u$, which is a constant multiple of 
 $ |x|^{2-n}$, is only superharmonic, and not harmonic.  Thus, in this case $v$ satisfies \eqref{necessary}, but not \eqref{sufficient}. 
 
\begin{proof}[Proof of Theorem \ref{main-riccati}]

To prove (i), suppose $u$ is a $p$-superharmonic solution to equation \eqref{eq1}. Let $\gamma=\frac{p-1-q}{p-1}$,  $v= \frac{1}{\gamma}u^{\gamma}$ and 
$$ u_k = \min\left (u, \, (\gamma k)^{\frac {1} {\gamma}}\right), 
 \quad v_k=\min (v, \, k),  \quad k=1, 2, \ldots .
 $$ 
 Notice that $v$ is $p$-superharmonic  since $x \longmapsto x^{\gamma}$ is concave and increasing.  We have
\begin{equation} \label{q-test}
\int_{\R^n} |Du|^{p-2}Du \cdot \nabla \phi \,  dx  = \int_{R^n} u^q \phi\, d \sigma, \quad \forall \,  \phi \in C_0^{\infty}(\R^n).
\end{equation}
By Theorem 3.15 in \cite{KKT}, $u$ is a (locally) renormalized solution to \eqref{eq1}. Therefore,
\begin{equation} \label{renormalized}
\int_{\R^n} |Du|^{p-2}Du \cdot \nabla (h(u)\phi) \,  dx  = \int_{R^n} u^q h(u) \phi\, d \sigma ,
\end{equation}
for all $\phi \in C_0^{\infty}(\R^n)$ and $h \in W^{1,\infty}(\R)$ with $h'$ having compact support.

Suppose $ \phi \in C_0^{\infty}(\R^n)$ and $h(u)= \frac{1}{u_k^q}$. Then 
$$\int_{\R^n} |Du|^{p-2}Du \cdot \nabla \left(\frac{\phi}{u_k^q}\right) \,  dx  = \int_{R^n} u^q \frac{\phi}{u_k^q} \, d \sigma .$$
Consequently, 
\begin{equation}
\begin{aligned}
& \int_{\R^n} |Du|^{p-2}Du \cdot \nabla \phi \, \frac{1}{u_k^q} \,  dx \\
= & \int_{R^n} u^q \frac{\phi}{u_k^q} \, d \sigma  +  q \int_{\R^n} |Du|^{p-2}Du \cdot \nabla u_k \, \frac{\phi}{u_k^{1+q}} \,  dx.
\end{aligned}
\end{equation}
Notice that $Du=(\gamma v)^{\frac{1}{\gamma}-1}Dv,$ and so 
$$|Du|^{p-1}=(\gamma v)^{\frac{q}{\gamma}}|Dv|^{p-1}.$$ 
Since $u$ is $p$-superharmonic,  
\begin{equation} \label{vkloc}
(\gamma v)^{\frac{q}{\gamma}}|Dv|^{p-1} \in L^1_{\text{loc}}(\R^n).
\end{equation}
From this it follows, 
\begin{equation}\label{v-eq}
\begin{aligned}
& \int_{\R^n} |Dv|^{p-2}Dv \cdot \nabla \phi \, \frac{(\gamma v )^{\frac{q}{\gamma}}}{(\gamma v_k )^{\frac{q}{\gamma}}} \,  dx \\  =& \int_{\R^n} (\gamma v )^{\frac{q}{\gamma}} \frac{\phi}{(\gamma v_k )^{\frac{q}{\gamma}}} \, d \sigma  +  {b}  \int_{\R^n} |Dv|^{p-2}Dv \cdot \nabla v_k \, \frac{(\gamma v )^{\frac{q}{\gamma}}\phi}{ v_k (\gamma v_k )^{\frac{q}{\gamma}}} \,  dx.
\end{aligned}
\end{equation}
Let $E=\text{supp}( \phi )$; then $v_1 \ge \delta_E >0$ a.e., and hence q.e., since $v_1$ is a positive superharmonic function. Notice that the sequence $\{v_k\}$ is increasing,  so that $v_k \ge \delta_E >0$ q.e. Consequently, 
$$|Dv|^{p-2}Dv \cdot \nabla \phi \, \frac{(\gamma v )^{\frac{q}{\gamma}}}{(\gamma v_k )^{\frac{q}{\gamma}}} \le \frac{||\nabla \phi||_{L^{\infty}(\R^n)}}{(\gamma \delta_E )^{\frac{q}{\gamma}}}  \,|Dv|^{p-1} (\gamma v )^{\frac{q}{\gamma}}\, \,  \text{ on } \,  E.$$
Using \eqref{vkloc} and the Dominated Convergence Theorem, we obtain
\begin{equation} \label{lim-C}
\int_{\R^n} |Dv|^{p-2}Dv \cdot \nabla \phi \, \frac{(\gamma v )^{\frac{q}{\gamma}}}{(\gamma v_k )^{\frac{q}{\gamma}}} \,  dx    \to \int_{\R^n} |Dv|^{p-2}Dv \cdot \nabla \phi \,  dx, 
\end{equation} 
 as $k \to \infty$, where the right-hand side is obviously finite.

Assuming temporarily that $\phi\ge 0$, we obtain from \eqref{v-eq}, 
\begin{equation} \label{compare}
\begin{aligned}
0 \le & {b}  \int_{\R^n} |Dv|^{p-2}Dv \cdot \nabla v_k \, \frac{(\gamma v )^{\frac{q}{\gamma}}\phi}{ v_k (\gamma v_k )^{\frac{q}{\gamma}}} \,  dx \\ 
\le & \int_{\R^n} |Dv|^{p-2}Dv \cdot \nabla \phi \, \frac{(\gamma v )^{\frac{q}{\gamma}}}{(\gamma v_k )^{\frac{q}{\gamma}}} \,  dx\le C, 
\end{aligned}
\end{equation} 
where by \eqref{lim-C}, $C$ does not depend on $k$. Clearly, 
$$0 \le  |Dv|^{p-2}Dv \cdot \nabla v_k \, \frac{(\gamma v )^{\frac{q}{\gamma}}\phi}{ v_k (\gamma v_k )^{\frac{q}{\gamma}}}\le  |Dv|^{p-2}Dv \cdot \nabla v_{k+1} \, \frac{(\gamma v )^{\frac{q}{\gamma}}\phi}{ v_{k+1} (\gamma v_{k+1} )^{\frac{q}{\gamma}}}.$$
Thus, using the Monotone Convergence Theorem and \eqref{compare}, we deduce 
$$\int_{\R^n} |Dv|^{p-2}Dv \cdot \nabla v_k \, \frac{(\gamma v )^{\frac{q}{\gamma}}\phi}{ v_k (\gamma v_k )^{\frac{q}{\gamma}}} \,  dx  \to\int_{\R^n}  \frac{|Dv|^p \, \phi}{ v} \,  dx 
\le \frac{C}{b}.
$$
 as $k \to \infty$. Hence, 
\begin{equation}\label{gradloc}
\frac{|Dv|^p}{ v} \in L^1_{\text{loc}}(\R^n, dx).
\end{equation}

Notice that, for all $\phi \in C^\infty_0(\R^n)$,  
$$(\gamma v )^{\frac{q}{\gamma}} \frac{|\phi|}{(\gamma v_k )^{\frac{q}{\gamma}}} \le \frac{||\phi||_{L^{\infty}}}{(\delta_E )^{\frac{q}{\gamma}}}   (\gamma v )^{\frac{q}{\gamma}}   \quad {\rm q.e.} \quad \text{on} \,\, E=\text{supp}( \phi ).$$
Since $\sigma$ is absolutely continuous with respect to the $p$-capacity, it follows that the preceding inequality holds on $E$  $d\sigma$-a.e.
Using the Dominated Convergence Theorem and the fact that $(\gamma v ) ^{\frac{q}{\gamma}} = u^q \in L^1_{\text{loc}}(\R^n, d\sigma),$ we obtain, for all $\phi \in C^\infty_0(\R^n)$,  
$$\int_{R^n} (\gamma v )^{\frac{q}{\gamma}} \frac{\phi}{(\gamma v_k )^{\frac{q}{\gamma}}} \, d \sigma  \to \int_{R^n} \phi \, d \sigma. $$
Clearly, 
$$ \left | \, |Dv|^{p-2}Dv \cdot \nabla v_k \, \frac{(\gamma v )^{\frac{q}{\gamma}}\phi}{ v_k (\gamma v_k )^{\frac{q}{\gamma}}} \, \right | \le \frac{|Dv|^p|\phi|}{ v}.$$ 
Using \eqref{gradloc} and the Dominated Convergence Theorem again, we obtain 
$$\int_{\R^n} |Dv|^{p-2}Dv \cdot \nabla v_k \, \frac{(\gamma v )^{\frac{q}{\gamma}}\phi}{ v_k (\gamma v_k )^{\frac{q}{\gamma}}} \,  dx \to 
\int_{\R^n} \frac{|Dv|^p\, \phi}{ v}  \,  dx  \text{\quad as } k \to \infty . $$
Therefore, letting $k \to \infty $ in \eqref{v-eq}, we deduce
$$\int_{\R^n} |Dv|^{p-2}Dv \cdot \nabla \phi  \, dx  =   {b}  \int_{\R^n} \frac{|Dv|^p\, \phi}{ v}  \,  dx  + \int_{R^n} \phi \, d \sigma, \quad \forall \, \phi \in C_0^{\infty}(\R^n). 
$$ 
Thus, $v$ is a $p$-superharmonic (and hence locally renormalized) solution to  \eqref{ric-eq}. Moreover, if both \eqref{Wolfinite} and \eqref{suffcond1} hold, then  by Theorem~\ref{mainthm} the minimal solution $u$ satisfies \eqref{two-sided}, and consequently 
$v$ satisfies both \eqref{lowerB} and \eqref{upperB}.

To prove (ii), suppose $v$ is a $p$-superharmonic solution to \eqref{ric-eq}.  Let  $\omega_k = -\Delta_p v_k$. Then   $v_k\in W^{1,p}_{{\rm loc}}(\R^n)\bigcap L^\infty(\R^n)$ is $p$-superharmonic, and 
\begin{equation}\label{e1}
-\Delta_p v_k = {b}  \frac{|\nabla v_k|^p}{v_k} +\sigma \, \chi_{v<k} +\tilde \omega_k,
 \end{equation}
where $\tilde \omega_k$ is a nonnegative measure in $\R^n$ supported on 
$\{v=k\}$.

We have 
$$u_k= (\gamma  v_k) ^{\frac{1}{\gamma}} \quad \text{and} \quad u_k\in W^{1,p}_{{\rm loc}}(\R^n)\cap L^\infty(\R^n),$$ 
since $v_k\in W^{1,p}_{{\rm loc}}(\R^n)\cap L^\infty(\R^n)$ and $\frac{1}{\gamma}=\frac{p-1}{p-1-q} > 1$. Let $\mu_k=-\Delta_p u_k$. Then it follows, 
\begin{equation}\label{eq2}
\mu_k=-\Delta_p u_k = -\Delta_p v_k \, (\gamma v_k)^{\frac q {\gamma}} - {b} \frac{|\nabla v_k|^p}{v_k}  \, (\gamma v_k)^{\frac q {\gamma}}  \ge 0.  
\end{equation}

Indeed, for any $\phi \in C^\infty_0(\R^n)$, 
\begin{equation*}
\begin{aligned}
& \int_{\R^n} \phi \, (\gamma v_k)^{\frac q {\gamma}}\, d \omega_k = 
\int_{\R^n}  \nabla (\phi \, (\gamma v_k)^{\frac q {\gamma}}) \cdot \nabla v_k \, |\nabla v_k|^{p-2}   \, dx \\
= & \int_{\R^n}   (\gamma v_k)^{\frac q {\gamma}} \nabla \phi    \cdot \nabla v_k \, |\nabla v_k|^{p-2}  \, dx  + {b}   \, \int_{\R^n}  (\gamma v_k)^{\frac q {\gamma}} \phi \, \frac{|\nabla v_k|^p}{v_k} \,  \,dx. 
\end{aligned}
\end{equation*}
Hence, 
\begin{equation*}
\begin{aligned}
\langle \phi, \mu_k\rangle  &= \int_{\R^n} \nabla \phi \, \cdot  \nabla ((\gamma v_k)^{\frac 1 {\gamma}}) \, |\nabla ((\gamma v_k)^{\frac 1 {\gamma}})|^{p-2}  \, dx \\
& = \int_{\R^n} \nabla \phi \,  \cdot  \nabla v_k \, |\nabla v_k|^{p-2} \, (\gamma v_k)^{\frac q {\gamma}} \, dx \\
& = \int_{\R^n} \phi \, (\gamma v_k)^{\frac q {\gamma}} \, d \omega_k 
- {b}  \, \int_{\R^n}    \phi  \, \frac{| \nabla v_k |^p}{v_k} \, (\gamma v_k)^{\frac q {\gamma}} \,  dx \\ 
& =\int_{\R^n} \phi \,  (\gamma v_k)^{\frac q {\gamma}} \  \chi_{v<k} \, d \sigma + \int_{\R^n} \phi \,  (\gamma v_k)^{\frac q {\gamma}} \,   d \tilde \omega_k,
\end{aligned}
\end{equation*}
where in the last expression we used \eqref{e1}. From the preceding estimates it follows that  
$
\langle \phi, \mu_k\rangle \ge 0$ if $\phi \ge 0$,   and consequently $u_k$ is $p$-superharmonic.

Clearly,  $u= (\gamma v)^{\frac 1 {\gamma}}<+\infty$-a.e., and  $u= \lim 
_{k \to +\infty} u_k$ is $p$-superharmonic in $\R^n$ as the limit of the 
increasing sequence of $p$-superharmonic functions $u_k$. 

Since $v$ is a $p$-superharmonic solution of the equation \eqref{ric-eq}, it follows that $v$ is a locally renormalized solution (see \cite{KKT}).
Then, for all $\phi \in C_0^{\infty}(\R^n)$ and $h \in W^{1,\infty}(\R)$ with $h'$ having compact support, we obtain 
\begin{equation} \label{ric-test}
\begin{aligned}
\int_{\R^n} |Dv|^{p-2}Dv \cdot \nabla  (h(v)\phi )\, dx& = {b}  \int_{\R^n} \frac{|D v|^p}{v}\, h(v) \phi\, dx \\& + \int_{\R^n} h(v) \phi \, d \, \sigma.
\end{aligned}
\end{equation}
Let  $\phi \in C_0^{\infty}(\R^n)$, $\phi \ge 0$. For $k >0$, set $h(v)= (\gamma v_k )^{\frac{q}{\gamma}}$. 
Then
$$\int_{\R^n} |Dv|^{p-2}Dv \cdot \nabla  ((\gamma v_k)^{\frac{q}{\gamma}}\phi )\, dx= {b}  \int_{\R^n} \frac{|D v|^p}{v}\, (\gamma v_k)^{\frac{q}{\gamma}}\phi \, dx + \int_{\R^n} (\gamma v_k)^{\frac{q}{\gamma}}\phi  \, d \, \sigma,$$
which yields 
\begin{equation*} 
\begin{aligned}
& \int_{\R^n} |Dv|^{p-2}Dv \cdot \nabla \phi \, (\gamma v_k)^{\frac{q}{\gamma}} \, dx + {b} \int_{\R^n} |Dv|^{p-2}Dv \cdot \nabla  v_k \, \frac{(\gamma v_k)\frac{q}{\gamma}}{v_k} \, \phi \, dx \\
& = {b}   \int_{\R^n} \frac{|D v|^p}{v}\, (\gamma v_k)^{\frac{q}{\gamma}}\phi \, dx + \int_{\R^n} (\gamma v_k)^{\frac{q}{\gamma}}\phi  \, d \, \sigma.
\end{aligned}
\end{equation*}
Hence,
\begin{equation}\label{usub}
\begin{aligned}
\int_{\R^n} |Dv|^{p-2}Dv \cdot \nabla \phi \, (\gamma v_k)^{\frac{q}{\gamma}} \, dx & = {b} \int_{v > k} \frac{|Dv|^p}{v} \, (\gamma v_k)^{\frac{q}{\gamma}} \, \phi \, dx  \\
&+ \int_{\R^n} (\gamma v_k)^{\frac{q}{\gamma}}\phi  \, d \, \sigma.  
\end{aligned}
\end{equation}
Consequently, 
\begin{equation}\label{vweak}
\begin{aligned}
\int_{\R^n} |Dv|^{p-2}Dv \cdot \nabla \phi \, (\gamma v_k)^{\frac{q}{\gamma}} \, dx & = {b}  \gamma ^{\frac{q}{\gamma}} k^{\frac{q(p-1)}{p-1-q}}\int_{v > k} \frac{|Dv|^p}{v} \, \phi \, dx\\  & + \int_{\R^n} (\gamma v_k)^{\frac{q}{\gamma}}\phi  \, d \, \sigma.
\end{aligned}
\end{equation}
Therefore,
\begin{equation} \label{v-super}
\int_{\R^n} |Dv|^{p-2}Dv \cdot \nabla \phi \, (\gamma v_k)^{\frac{q}{\gamma}} \, dx  \ge \int_{\R^n} (\gamma v_k)^{\frac{q}{\gamma}}\phi  \, d \, \sigma.
\end{equation}
Note that $ Du= (\gamma v)^{\frac{q}{p-1-q}} D v,$ so that $ |Du|^{p-1}= (\gamma v)^{\frac{q}{\gamma}} |D v|^{p-1}$, and 
$$ \left \vert  \, |Dv|^{p-2}Dv \cdot \nabla \phi \, (\gamma v_k)^{\frac{q}{\gamma}} \, \right \vert  \le |\nabla \phi| \, |Dv|^{p-1} (\gamma v)^{\frac{q}{\gamma}}  \le ||\nabla \phi||_{L^{\infty}(\R^n)}|D u|^{p-1}.$$
Notice that $|Du|^{p-1} \in L^1_{\text{loc}}(\R^n, dx)$.
Using the Dominated Convergence Theorem, we obtain
\begin{equation*}
\begin{aligned}
\int_{\R^n}|Dv|^{p-2}Dv \cdot \nabla \phi \, (\gamma v_k)^{\frac{q}{\gamma}} \, dx & \to \int_{\R^n} |Dv|^{p-2}Dv \cdot \nabla \phi \, (\gamma v)^{\frac{q}{\gamma}} \, dx \\& = \int_{\R^n} |Du|^{p-2}Du \cdot \nabla \phi \,  dx.
\end{aligned}
\end{equation*}
From this and \eqref{vweak}, we have 
$$ {b}  \gamma ^{\frac{q}{\gamma}} \int_{v > k} \frac{|Dv|^p}{v} \, \phi \, dx \le {k^{-\frac{q(p-1)}{p-1-q}}}C(u,\phi) < \infty.$$
Therefore,
\begin{equation} \label{v-necessary}
||v||_{L^{\frac{q(p-1)}{p-1-q}, \infty}(\phi \frac{|Dv|^p}{v}dx)} < \infty.
\end{equation}
Using \eqref{v-super} and the Monotone Convergence Theorem, we deduce 
$$ \int_{\R^n} |Du|^{p-2}Du \cdot \nabla \phi \,  dx  \ge \int_{R^n} u^q \phi \, d \sigma, \quad \forall \,  \phi \in C_0^{\infty}(\R^n), \quad \phi \ge 0. $$ Moreover, $u$ is $p$-superharmonic on $\R^n$. This means that $u$ is a supersolution of \eqref{eq1} in the (locally) renormalized sense (see \cite{KKT}). By Theorem~\eqref{mainthm}, $u$ satisfies the lower bound in \eqref{two-sided}, and consequently 
$v$ satisfies \eqref{lowerB}.

It remains to prove (iii). Suppose additionally that 
$$\int_{B} v^{\frac{q(p-1)}{p-1-q}} \frac{|Dv|^p}{v} \,  dx < \infty,$$ 
for every ball $B$. 
Then $$\int_{v > k} \frac{|Dv|^p}{v} \, (\gamma v_k)^{\frac{q}{\gamma}} \, \phi \, dx  \to 0 $$ by the Dominated Convergence Theorem.
Letting $k \to \infty$ in \eqref{usub}, we deduce 
$$ \int_{\R^n} |Du|^{p-2}Du \cdot \nabla \phi \,  dx  = \int_{R^n} u^q \phi \, d \sigma, \quad \forall \,  \phi \in C_0^{\infty}(\R^n). $$
Thus,  $u$ is a $p$-superharmonic, and hence a locally renormalized, solution to \eqref{eq1}. This completes the proof of Theorem~\ref{main-riccati}. \end{proof}

As a corollary of Theorem~\ref{main-riccati}, one can characterize the existence of finite energy solutions $v\in L^{1, p}_{0}(\R^n)$ to \eqref{ric-eq}. It is easy to see that such solutions exist if and only if ${b} <1$, and  $\sigma\in L^{-1, p}(\R^n)$, i.e.,  
$\int_{\R^n} \left(\mathbf{W}_{1, p} \sigma\right) \, d \sigma < \infty$.


\begin{thebibliography}{DMMOP99}


\bibitem[AGP11]{AGPW} B. Abdellaoui, D. Giachetti, I. Peral, and M. Walias, \emph{Elliptic problems with nonlinear terms depending on the gradient and singular on the boundary,} Nonlin. Analysis \textbf{74} (2011), 1355--1371.

\bibitem[ABV10]{AHBV} H. Abdul-Hamid and M.-F. Bidaut-V\'{e}ron, \emph{On the connection between two quasilinear elliptic problems with source terms of order 0 or 1,} Commun. Contemp. Math. \textbf{12} (2010), 727--788.

\bibitem[AH96]{AH}  {\sc D.~R. Adams and L.~I. Hedberg,} \emph{Function Spaces and Potential Theory},
  Grundlehren der math. Wissenschaften {\bf 314}, Berlin--Heidelberg--New York, Springer, 1996.


\bibitem[ABL10]{ABLP} D. Arcoya, L. Boccardo, T. Leonori, and  A. Porretta, \emph{Some elliptic problems with singular natural growth lower order terms,} J. Diff. Eqs. \textbf{249} (2010), 2771--2795.
 

\bibitem[BiVe03]{BV} M.-F. Bidaut-V\'{e}ron, \emph{Removable singularities and existence for a quasilinear equation
with absorption or source term and measure data,} Adv. Nonlin. Studies \textbf{3} (2003), 25--63.


 
\bibitem[BoOr96]{BO1} L.~ Boccardo and L.~ Orsina, \emph{Sublinear elliptic equations in $L^s$,} Houston Math. J. \textbf{20} (1994), 99--114. 


\bibitem[BK92]{BK} H. Brezis and S. Kamin,  \emph{Sublinear elliptic equation on $\mathbb{R}^n$,} Manuscr. Math. \textbf{74} (1992),  87--106. 

\bibitem[BO86]{BrOs} H. Brezis and L. Oswald,  \emph{Remarks on sublinear elliptic equations,} Nonlin. Analysis  \textbf{10} (1986),  55--64. 



\bibitem[CV14]{CV1} D. T. Cao and I.~E. Verbitsky, \emph{Finite energy solutions 
of quasilinear elliptic equations with sub-natural growth terms}, Calc. Var. PDE (2014), DOI  10.1007/s00526-014-0722-0.

\bibitem[DMM99]{DMMOP}
G.~Dal Maso, F.~Murat, L.~Orsina and A.~Prignet, \emph{Renormalized solutions of elliptic equations with general measure data},
Ann. Scuola Norm. Super. Pisa \textbf{28} (1999), 741--808.



\bibitem[FM00]{FM}
V.~Ferone and F.~Murat, \emph{Nonlinear problems having natural growth in the
  gradient: an existence result when the source terms are small}, Nonlin. 
  Analysis \textbf{42} (2000), 1309--1326.

 

\bibitem[GM09]{GM}
D.~Giachetti and F.~Murat, \emph{An elliptic problem with a lower order term having singular behavior}, Boll. Unione Mat. Ital. (9)  \textbf{2} 
(2009), 349--370.



\bibitem[HW83]{HW}
L.~I. Hedberg and T.~Wolff, \emph{Thin sets in nonlinear potential theory},
  Ann. Inst. Fourier (Grenoble) \textbf{33} (1983), 161--187.


  
\bibitem[HKM06]{HKM}
J.~Heinonen, T.~Kilpel\"{a}inen, and O.~Martio, \emph{Nonlinear Potential
  Theory of Degenerate Elliptic Equations}, Dover Publications, 2006 (unabridged republ. of 1993 edition, Oxford Universiy Press).
	
\bibitem[JV10]{JV1} B.~J.~Jaye and I. E.~Verbitsky, \emph{The fundamental solution of nonlinear operators with natural growth terms,} Ann. Scuola Norm. Super. Pisa \textbf{12} (2013), 93--139.

\bibitem[JV12]{JV2} B.~J.~Jaye and I. E. ~Verbitsky, \emph{Local and global behaviour of solutions
to nonlinear equations with natural growth terms,} Arch. Rational Mech. Anal. \textbf{204} (2012), 627--681.



\bibitem[Kil02]{K2} T.~Kilpel\"{a}inen, \emph{p-Laplacian type equations involving measures,}  Proc. ICM, Vol. III, 167--176, Beijing, 2002. 

	

\bibitem[KKT09]{KKT} T.~Kilpel\"{a}inen, T.~Kuusi and A.~Tuhola-Kujanp\"{a}\"{a}, \emph{Superharmonic functions are locally renormalized solutions,} Ann. Inst. H. Poincar\'{e} Anal. Non Lin\'{e}aire \textbf{28} (2011),  775--795.  

\bibitem[KM92]{KM1}
T.~Kilpel\"{a}inen and J.~Mal\'y, \emph{Degenerate elliptic equations with
  measure data and nonlinear potentials}, Ann. Scuola Norm. Super. Pisa 
  \textbf{19} (1992), 591--613. 
	
\bibitem[KM94]{KM2}
T.~Kilpel\"{a}inen and J.~Mal\'y, \emph{The {W}iener test and potential estimates for
  quasilinear elliptic equations}, Acta Math. {\bf 172}  (1994), 137--161.




\bibitem[Kra64]{Kra} M. A. Krasnoselskii, \emph{Positive Solutions of Operator Equations,}  P. Noordhoff Ltd., Groningen, 1964. 

\bibitem[KuMi14]{KuMi} T.~Kuusi and G. Mingione, \emph{Guide to nonlinear potential estimates,} Bull. Math. Sci. \textbf{4} (2014), 1--82. 


\bibitem[Lab02]{Lab} D.~Labutin, \emph{Potential estimates for a class of fully nonlinear
elliptic equations}, Duke Math. J. \textbf{111} (2002), 1--49. 




\bibitem[MZ97]{MZ1}
J.~Mal\'{y} and W.~Ziemer, \emph{Fine Regularity of Solutions of Elliptic Partial
  Differential Equations}, Math. Surveys Monogr. \textbf{51},
  Amer. Math. Soc., Providence, RI, 1997.

  
\bibitem[Maz11]{Maz11} V. Maz'ya, \emph{Sobolev Spaces, with Applications to Elliptic Partial Differential Equations,} 2nd, Augm. Ed.,  Grundlehren der math. Wissenschaften \textbf{342}, Springer, Berlin, 2011.



\bibitem[PV08]{PV1}N. C. Phuc and I. E. Verbitsky, \emph{Quasilinear and Hessian equations of Lane-Emden type,} Ann. Math. \textbf{168} (2008), 859--914. 



\bibitem[TW99]{TW0} N. S. Trudinger and X. J. Wang, \emph{Hessian measures II}, 
        Ann. Math.  \textbf{150} (1999),  579--604.


\bibitem[TW02]{TW1} N. S. Trudinger and X. J. Wang, \emph{On the weak continuity of elliptic operations and applications to potential theory,} Amer. J. Math. \textbf{124} (2002), 369--410.


\bibitem[Ver99]{V1} I. E. Verbitsky, \emph{Nonlinear potentials and trace inequalities,}   
Oper. Theory Adv. Appl. \textbf{110} (1999), 323--343.



\end{thebibliography}
\end{document}